\documentclass[reqno, 12pt, superscriptaddress, twoside, tightenlines, notitlepage]{revtex4-1} 

\usepackage[utf8]{inputenc} 
\usepackage{tikz}
\usepackage{geometry} 
\usepackage{youngtab}
\usepackage{graphicx} 
\usepackage{amsmath}
\usepackage{amssymb}
\usepackage{amsthm}
\newtheorem{theorem}{Theorem}
\newtheorem{proposition}{Proposition}
\newtheorem{lemma}{Lemma}
\newtheorem{corollary}{Corollary}
\newtheorem{definition}{Definition}

\begin{document}
\title{Soliton cellular automaton associated with $D_n^{(1)}$-crystal $B^{2,s}$}
\author{Kailash C. Misra}
\email{misra@ncsu.edu}
\affiliation{
Department of Mathematics\\ North Carolina State University\\ 
Raleigh, NC 27695-8205}
\author{Evan A. Wilson}
\email{wilsonea@ime.usp.br}
\affiliation{
Department of Mathematics\\ North Carolina State University\\ 
Raleigh, NC 27695-8205}
\affiliation{
Instituto de Matem\'{a}tica e Estat\'{i}stica\\Universidade de S\~{a}o Paulo}
\begin{abstract}
\noindent\normalsize{A solvable vertex model in ferromagnetic regime gives rise to a soliton cellular automaton which is a discrete dynamical system in which site variables take on values in a finite set. We study the scattering of a class of soliton cellular automata associated with the $U_q(D_n^{(1)})$-perfect crystal $B^{2,s}$. We calculate the combinatorial $R$ matrix for all elements of $B^{2,s} \otimes  B^{2,1}$. In particular, we show that the scattering rule for our soliton cellular automaton can be identified with the combinatorial $R$ matrix for $U_q(A_1^{(1)}) \oplus U_q(D_{n-2}^{(1)})$-crystals.}
\end{abstract}

\maketitle

\section{Introduction} 
A cellular automaton is a dynamical system in which points in the one-dimensional space lattice are assigned discrete values in a finite set which evolve according to a deterministic rule.  Soliton cellular automata (SCA) are a  kind of cellular automata which possess stable configurations analogous to solitons in integrable partial differential equations. Solitons move with constant velocity proportional to the length if there is no collision. After collision their lengths are preserved but the phases are shifted. Moreover, SCAs have many conserved quantities.
The simplest example of a SCA is the Takahashi-Satsuma's automaton \cite{TS} where the site variables take on two values $\{0, 1\}$. This SCA gives rise to a nonlinear dynamical system. Subsequently, this was generalized to other SCAs where the site variables take on more than two values (\cite{T}, \cite{TNS}, \cite{TTM}). It was then found that these systems can be described by perfect crystals \cite{KMN} for the quantum affine algebra $U_q'(A_{n}^{(1)})$ (\cite{FOY}, \cite{HHIKTT}). In fact, it was shown in \cite{FOY} that the phase shift of the $U_q'(A_n^{(1)})$ SCA was related to the energy function of $U_q'(A_{n}^{(1)})$ perfect crystals. In \cite{HKT}, a class of SCAs associated with the perfect crystals for the non-exceptional quantum affine Lie algebras given in \cite{KKM} was constructed.  In \cite{Yd2} (resp. \cite{MOW}) the SCA associated with the $U_q'(D_4^{(3)})$ (resp. $U_q'(G_2^{(1)})$) perfect crystals given in \cite{KMOY} (resp. \cite{MMO}) have been constructed. In this formulation the time evolution operator is given by the row-to-row transfer matrix of integrable vertex models at $q=0$, represented as a product of the combinatorial $R$ matrix of perfect crystals.   Using this approach, the scattering rules for the SCA associated with the perfect crystals given in \cite{KKM} were determined in \cite{HKOTY3}. Among these SCAs, the SCA associated with the quantum affine algebra $U_q'(D_{n}^{(1)})$ is of fundamental importance, because the other ones can be imbedded into this one \cite{KTT}.

It is known that the perfect crystals for quantum affine Lie algebras are crystals for certain Kirillov-Reshetikhin (KR) modules \cite{KR}. Fourier, Okado and Schilling (\cite{FOS1},\cite{FOS2}) have proved the existence of perfect crystals for quantum affine algebras of classical types conjectured in (\cite{HKOTY1}, \cite{HKOTY2}). In particular, in \cite{FOS1} explicit descriptions of KR-crystals $B^{r,s}$ for any Dynkin node $ r \not= 0$ and any positive integer $s$ are given for each quantum affine algebras of classical types. In \cite{FOS2} it has been shown that if $s$ is a multiple of $t_r= {\rm max}\{1, 2/(\alpha_r, \alpha_r)\}$, then the KR-crystal $B^{r, s}$ is perfect of level $s/t_r$. So far all SCAs constructed are associated with the perfect crystal $B^{1, s}$ with the exception of the one in \cite{M} where the SCA associated with the $U_q'(D_n^{(1)})$ perfect crystal $B^{n,s}$ has been constructed. It is conjectured that the scattering rule for the SCA associated with a perfect crystal $B^{r,s}$ will be given by the $R$ matrix for the crystal of the quantum affine algebra 
$U_q'(g'^{(1)}) \oplus U_q'(g''^{(1)})$, where $g'$ and $g''$ are the simple Lie algebras associated with the two connected components obtained after removing the Dynkin node $r$ from the finite Dynkin diagram.

In this paper, using the realizations of $B^{2,s}$ given in \cite{FOS1} (see also \cite{SS}) we compute the combinatorial $R$ matrix $B^{2,s}\otimes B^{2,1} \tilde{\to}B^{2,1}\otimes B^{2,s}$ for $U_q'(D_n^{(1)})$, then use it to find the scattering rule of the SCA associated with 
$B^{2,s}$.  We see that solitons of length $s$ are parameterized by the $U_q(A_1^{(1)}) \oplus U_q(D_{n-2}^{(1)})$-crystal $\widehat{B}^{1,s} \times B^{1,s}$. Furthermore, when two solitons collide, the scattering rule is described by the combinatorial $R$ matrix $\widehat{B}^{1,s_1} \otimes B^{1,s_2}\tilde{\to} B^{1,s_2}\otimes \widehat{B}^{1,s_1}, s_1>s_2$. Thus we have shown that the above conjecture holds in this case.

\section{Preliminaries}
The infinite dimensional Kac-Moody algebra $\mathfrak{g}=D_n^{(1)}$ is determined by the following Dynkin diagram:
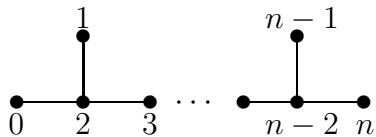
\begin{center}
\begin{picture}(155,30)
	
	\put(20,5){\circle*{5}}
	\put(45,5){\circle*{5}}
	\put(45,30){\circle*{5}}
	\put(70,5){\circle*{5}}
	\put(105,5){\circle*{5}}
	\put(125,5){\circle*{5}}
	\put(125,30){\circle*{5}}
	\put(150,5){\circle*{5}}
	\put(20,5){\line(1,0){25}}
	\put(45,5){\line(0,1){25}}
	\put(45,5){\line(1,0){25}}
	\put(105,5){\line(1,0){25}}
	\put(125,5){\line(0,1){25}}
	\put(125,5){\line(1,0){25}}
	
	\put(17,-7){0}
	\put(42,-7){2}
	\put(79,2){$\cdots$}
	\put(42,33){1}
	\put(67,-7){3}
	\put(113,33){$n-1$}
	\put(113,-7){$n-2$}
	\put(147,-7){$n$}
\end{picture}
\end{center}
We let $P=\text{span}_{\mathbb{Z}}\{\Lambda_0,\Lambda_1,\dots, \Lambda_n\}$ be the weight lattice, $\{\alpha_0,\alpha_1, \dots, \alpha_n\}$ be the set of simple roots, and $\{h_0,h_1,\dots, h_n\}$ be the set of simple coroots.
\subsection{Crystals}
In this section we give the basic definitions regarding crystals.  Our exposition follows \cite{HK}.  Let $I=\{0,1,2,\dots n\}$.  
\begin{definition}
A \emph{crystal} associated with $U_q(\mathfrak{g})$ is a set $B$ together with maps $\text{wt}:B\to P, \tilde{e}_i, \tilde{f}_i : B\to B\cup \{0\},$ and $\varepsilon_i, \varphi_i:B\to \mathbb{Z}\cup \{-\infty\}$, for $i\in I$ satisfying the following properties:
\begin{enumerate}
\item $\varphi_i(b)=\varepsilon_i(b)+\langle h_i, \text{wt}(b)\rangle$ for all $i \in I,$
\item $\text{wt}(\tilde{e}_ib)=\text{wt}(b)+\alpha_i$ if $\tilde{e}_ib\in B$,
\item $\text{wt}(\tilde{f}_ib)=\text{wt}(b)-\alpha_i$ if $\tilde{f}_ib\in B$,
\item $\varepsilon_i(\tilde{e}_ib)=\varepsilon_i(b)-1, \varphi_i(\tilde{e}_ib)=\varphi_i(b)+1$ if $\tilde{e}_ib\in B,$
\item $\varepsilon_i(\tilde{f}_ib)=\varepsilon_i(b)+1, \varphi_i(\tilde{f}_ib)=\varphi_i(b)-1$ if $\tilde{f}_ib\in B,$
\item $\tilde{f}_ib=b'$ if and only if $b=\tilde{e}_ib'$ for $b,b'\in B, i\in I,$
\item if $\varphi_i(b)=-\infty$ for $b\in B,$ then $\tilde{e}_ib=\tilde{f}_ib=0.$
\end{enumerate}
\end{definition}
A crystal $B$ can be regarded as a colored, oriented graph by defining
$$
b\stackrel{i}{\to} b' \iff \tilde{f}_ib=b'.
$$

\begin{definition}\label{tensor}The \emph{tensor product} $B_1\otimes B_2$ of crystals $B_1$ and $B_2$ is the set $B_1\times B_2$ together with the following maps:
\begin{enumerate}
\item $\text{wt}(b_1\otimes b_2)=\text{wt}(b_1)+\text{wt}(b_2),$
\item $\varepsilon_i(b_1\otimes b_2)=\max(\varepsilon_i(b_1),\varepsilon_i(b_2)-\langle h_i,\text{wt}(b_1)\rangle),$
\item $\varphi_i(b_1\otimes b_2)=\max(\varphi_i(b_2),\varphi_i(b_1)+\langle h_i,\text{wt}(b_2)\rangle),$
\item $\tilde{e}_i(b_1\otimes b_2)=\left \{ 
	\begin{array}{l}\tilde{e}_ib_1\otimes b_2, \text{ if } \varphi_i(b_1)\geq \varepsilon_i(b_2),\\
		b_1\otimes \tilde{e}_ib_2, \text{ if } \varphi_i(b_1)< \varepsilon_i(b_2),
	\end{array}\right .$
\item $\tilde{f}_i(b_1\otimes b_2)=\left \{ 
	\begin{array}{l}\tilde{f}_ib_1\otimes b_2, \text{ if } \varphi_i(b_1)> \varepsilon_i(b_2),\\
		b_1\otimes \tilde{f}_ib_2, \text{ if } \varphi_i(b_1)\leq \varepsilon_i(b_2),
	\end{array}\right .$
\end{enumerate}
where we write $b_1\otimes b_2$ for $(b_1,b_2)\in B_1\times B_2$, and understand $b_1\otimes 0=0\otimes b_2=0.$
\end{definition}
$B_1\otimes B_2$ is a crystal, as can easily be shown.

\begin{definition}
Let $B_1$ and $B_2$ be $U_q(\mathfrak{g})$-crystals.  A \emph{crystal isomorphism} is a bijective map $\Psi:B_1\cup \{0\}\to B_2\cup \{0\}$ such that
\begin{enumerate}
\item $\Psi(0)=0,$
\item if $b\in B_1$ and $\Psi(b)\in B_2,$ then $\text{wt}(\Psi(b))=\text{wt}(b), \varepsilon_i(\Psi(b))=\varepsilon_i(b),\varphi_i(\Psi(b))=\varphi_i(b)$ for all $i\in I,$
\item if $b,b'\in B_1, \Psi(b), \Psi(b')\in B_2$ and $\tilde{f}_ib=b',$ then $\tilde{f}_i\Psi(b)=\Psi(b')$ and $\Psi(b)=\tilde{e}_i\Psi(b')$ for all $i\in I.$
\end{enumerate}
\end{definition}

\subsection{$D_n^{(1)}$-crystal $B^{1,s}$}
Let $B^{1,s}:=\{b=(x_1,x_2,\dots, x_n,\bar{x}_n,\bar{x}_{n-1},\dots, \bar{x}_1)\in \mathbb{Z}_{\geq 0}^{2n}| s(b):=\sum_{i=1}^n x_i+\sum_{i=1}^{n}\bar{x}_i=s, x_n=0 \text{ or } \bar{x}_n=0\}$ and define
\begin{eqnarray*}
\tilde{e}_0 b &=&\left \{ \begin{array}{c} (x_1, x_2-1, \dots, \bar{x}_2, \bar{x}_1+1 ) \text{ if }x_2>\bar{x}_2,\\
	(x_1-1, x_2, \dots, \bar{x}_2+1, \bar{x}_1 ) \text{ if }x_2\leq \bar{x}_2, \end{array} \right .\\
\tilde{e}_{n} b &=&\left \{ \begin{array}{c}(x_1, \dots, x_{n}+1, \bar{x}_n, \bar{x}_{n-1}-1, \dots, \bar{x}_1) \text{ if }x_n\geq0, \bar{x}_n=0,\\ 
	(x_1, \dots, x_{n-1}+1, x_n, \bar{x}_{n}-1, \dots, \bar{x}_1) \text{ if }x_n=0, \bar{x}_n>0,\end{array}\right .\\
\tilde{e}_{i} b &=&\left \{\begin{array}{c}(x_1, \dots, x_{i}+1, x_{i+1} -1, \dots, \bar{x}_1) \text{ if }x_{i+1}>\bar{x}_{i+1},\\ 
	(x_1, \dots,\bar{x}_{i+1}+1, \bar{x}_{i}-1, \dots, \bar{x}_1) \text{ if }x_{i+1}\leq\bar{x}_{i+1}.\end{array}\right .1\leq i \leq n-1\\
\end{eqnarray*}
\begin{eqnarray*}
\tilde{f}_0 b &=& \left \{ \begin{array}{c}(x_1, x_2+1, \dots, \bar{x}_2, \bar{x}_1-1 ) \text{ if }x_2\geq \bar{x}_2,\\
	(x_1+1, x_2, \dots, \bar{x}_2-1, \bar{x}_1 ) \text{ if }x_2 < \bar{x}_2,\end{array} 
	\right .\\
\tilde{f}_{n} b &=& \left \{\begin{array}{c}(x_1, \dots, x_{n}-1, \bar{x}_n, \bar{x}_{n-1}+1, \dots, \bar{x}_1) \text{ if }x_n>0, \bar{x}_n=0,\\ 
	(x_1, \dots, x_{n-1}-1, x_n, \bar{x}_{n}+1, \dots, \bar{x}_1) \text{ if }x_n=0, \bar{x}_n\geq0, \end{array} \right .\\
\tilde{f}_{i} b &=&\left \{ \begin{array}{c}(x_1, \dots, x_{i}-1, x_{i+1} +1, \dots, \bar{x}_1) \text{ if }x_{i+1}\geq\bar{x}_{i+1},\\ 
	(x_1, \dots,\bar{x}_{i+1}-1, \bar{x}_{i}+1, \dots, \bar{x}_1) \text{ if }x_{i+1}<\bar{x}_{i+1}. \end{array} \right .1\leq i \leq n-1
\end{eqnarray*}
If $x_i<0$ or $\bar{x}_i<0$ in $b'=\tilde{e}_i(b)$ or $\tilde{f}_i(b)$ then $b'$ is understood to be $0$.
\begin{eqnarray*}
\text{wt}(b)&=&(\bar{x}_1-x_1+\bar{x}_2-x_2)\Lambda_0+\sum_{i=1}^{n-2}(x_i-\bar{x}_i+\bar{x}_{i+1}-x_{i+1})\Lambda_i\\
            &&+\:(x_{n-1}-\bar{x}_{n-1}+\bar{x}_n-x_n)\Lambda_{n-1}\\
	&&+\:(x_{n-1}-\bar{x}_{n-1}+x_n-\bar{x}_n)\Lambda_{n},\\
\varphi_0(b) &=& \bar{x}_1+(\bar{x}_2-x_2)_+, \qquad \varepsilon_0(b)=x_1+(x_2-\bar{x}_2)_+,\\ 
\varphi_i(b) &=& x_i+(\bar{x}_{i+1}-x_{i+1})_+ \text{ for }1\leq i\leq  n-2,\\ 
\varepsilon_i(b)&=&\bar{x}_i+(x_{i+1}-\bar{x}_{i+1})_+ \text{ for }1\leq i\leq  n-2,\\
\varphi_{n-1}(b) &=& x_{n-1}+\bar{x}_n, \qquad \varepsilon_{n-1}(b)=\bar{x}_{n-1}+x_n,\\ 
\varphi_{n}(b) &=& x_{n-1}+x_n, \qquad \varepsilon_{n}(b)=\bar{x}_{n-1}+\bar{x}_n,
\end{eqnarray*}
where $(n)_+:=\max(n,0)$.
Then we have the following:
\begin{theorem}[\cite{KKM},\cite{KMN}]
The maps $\tilde{e}_i, \tilde{f}_i, \varepsilon_i,\varphi_i,\text{wt}$ define a $U_q'(D_n^{(1)})$-crystal structure on $B^{1,s}$.
\end{theorem}
We associate the element $(x_1,x_2,\cdots, \overline{x}_1)\in B^{1,s}$ with the tableau: 
$$\begin{array}{l @{} l @{} l @{} l @{} l @{} l}\underbrace{\begin{array}{|c|c|c|c|}\hline 1 & 1 & \cdots & 1 \\ \hline\end{array}}_{x_1}& 
\underbrace{\begin{array}{c|c|c|c|}\hline  2 & 2 & \cdots & 2\\ \hline \end{array}}_{x_2}&\begin{array}{|c|}\hline \cdots\\ \hline \end{array}&\underbrace{\begin{array}{c|c|c|c|}\hline  \overline{1} & \overline{1} & \cdots & \overline{1}\\ \hline \end{array}}_{\overline{x}_1}\end{array}.$$

\subsection{Perfect crystal $B^{2,s}$ for $D_n^{(1)}$}
In this section we review the perfect crystals $B^{2,s},s\geq 1$ corresponding to the 2-node of the Dynkin diagram of $D_n^{(1)}$.  The existence of crystal bases $B^{2,s},s\geq 1$ was proven in \cite{OS} and the combinatorial realization was given in \cite{FOS1} (see also \cite{SS}, \cite{KN}).

For $D_n^{(1)}$, define the alphabet $\mathcal{B}=\{1,2,\dots, n-1,n,\overline{n},\overline{n-1},\dots, \overline{2},\overline{1}\}$.  Define the following partial ordering on $\mathcal{B}$:

$$
1 < 2 < \cdots < n-1 <\begin{array}{c }n \\
						\overline{n}  \end{array}
< \overline{n-1} < \cdots < \overline{2} < \overline{1}.$$
For $D_n$, define the set:
$$
B(k\Lambda_2)=\left \{ \begin{array}{l|l}
\begin{array}{l}T=\begin{array}{|c|c|c|c|}
\hline
T_{1,1}&T_{1,2}&\cdots&T_{1,k}\\
\hline
T_{2,1}&T_{2,2}&\cdots&T_{2,k}\\
\hline
\end{array}\\
T_{i,j}\in \mathcal{B}
\end{array}
&
\begin{array}{l}
T_{i,j}\leq T_{i,j+1}, \\
\qquad i=1,2, j=1,2,\dots k-1\\
T_{1,j}<T_{2,j} \text{ or }T_{1,j}=\overline{n}, T_{2,j}=n,\\
\qquad  j=1,2,\dots k,\\
\text{no column of the form }\begin{array}{|c|}
\hline 1\\
\hline \overline{1}\\
\hline
\end{array} \text{ occurs,}\\
\text{no configuration of the form }\\
\qquad \begin{array}{|c|c|} \hline a&a\\
\hline * & \overline{a}\\
\hline
\end{array}
\text{ or }
\begin{array}{|c|c|} \hline a&*\\
\hline \overline{a}& \overline{a}\\
\hline
\end{array}
\text{ occurs, and}\\
\text{no configuration of the form }\\
\qquad \begin{array}{|c|c|} \hline n-1&n\\
\hline n& \overline{n-1}\\
\hline
\end{array}
\text{ or }
\begin{array}{|c|c|} \hline n-1&\overline{n}\\
\hline \overline{n}& \overline{n-1}\\
\hline
\end{array}\\
\text{occurs.}
\end{array}
\end{array}
\right \}
$$
The set $B(k\Lambda_2)$ becomes a $U_q'(D_n)$-crystal by considering it as a tensor product (see Definition \ref{tensor}) under the following reading:
$$
T=\begin{array}{|c|c|c|c|}
\hline
T_{1,1}&T_{1,2}&\cdots&T_{1,k}\\
\hline
T_{2,1}&T_{2,2}&\cdots&T_{2,k}\\
\hline
\end{array}
$$
corresponds to
$$
T_{1,k}\otimes T_{2,k}\otimes T_{1,k-1}\otimes T_{2,k-1}\otimes \cdots \otimes T_{1,1}\otimes T_{2,1}.
$$
The maps $\tilde{e}_i,\tilde{f}_i,\varepsilon_i,\varphi_i, i\in I$ and $\overline{\text{wt}}$ are given for an individual letter $b\in \mathcal{B}$ by:
\begin{eqnarray}
\tilde{e}_i(b)&=&\left \{ \begin{array}{ll}
i, &\text{if }b=i+1,\\
\overline{i+1}, &\text{if }b=\overline{i},\\
0, & \text{otherwise}.
\end{array}\right .i=1,2,\dots, n-1,\\
\tilde{e}_n(b)&=&\left \{ \begin{array}{ll}
n, &\text{if }b=\overline{n-1},\\
n-1, &\text{if }b=\overline{n},\\
0, & \text{otherwise}.
\end{array}\right .\\
\tilde{f}_i(b)&=&\left \{ \begin{array}{ll}
i+1, &\text{if }b=i,\\
\overline{i}, &\text{if }b=\overline{i+1},\\
0, & \text{otherwise}.
\end{array}\right .i=1,2,\dots, n-1,\\
\tilde{f}_n(b)&=&\left \{ \begin{array}{ll}
\overline{n-1}, &\text{if }b=n,\\
\overline{n}, &\text{if }b=n-1,\\
0, & \text{otherwise}.
\end{array}\right .\\
\varepsilon_i(b)&=&\left \{ \begin{array}{ll}
1, &\text{if }b=i+1,\overline{i},\\
0, & \text{otherwise}.
\end{array}\right .i=1,2,\dots, n-1,\\
\varepsilon_n(b)&=&\left \{ \begin{array}{ll}
1, &\text{if }b=\overline{n-1},\overline{n},\\
0, & \text{otherwise}.
\end{array}\right .\\
\varphi_i(b)&=&\left \{ \begin{array}{ll}
1, &\text{if }b=i,\overline{i+1},\\
0, & \text{otherwise}.
\end{array}\right .i=1,2,\dots, n-1,\\
\varphi_n(b)&=&\left \{ \begin{array}{ll}
1, &\text{if }b=n-1,n,\\
0, & \text{otherwise}.
\end{array}\right .\\
\overline{\text{wt}}(b)&=&\left \{ \begin{array}{ll}
\Lambda_{1}, & \text{if }b=1,\\
\Lambda_{i}-\Lambda_{i-1}, & \text{if }b=2,3,\dots, n-2,n,\\
\Lambda_{n-1}+\lambda_{n}-\Lambda_{n-2}, &\text{if }b=n-1, 
\end{array}\right .\\
\overline{\text{wt}}(\overline{b})&=&-\overline{\text{wt}}(b), \text{ in any other case}.\nonumber
\end{eqnarray}

We then have the following:

\begin{theorem}[\cite{OS},\cite{FOS1}]
$B^{2,s}, s\geq 0$ is a $D_n^{(1)}$-crystal, and, forgetting the 0-arrows, is isomorphic to the $D_n$-crystal $\bigoplus_{k=0}^{s}B(k\Lambda_2)$.
\end{theorem}

The 0-action in $B^{2,s}$ was described in \cite{SS} and later in \cite{FOS1} for arbitrary $B^{r,s}$.  It is given in terms of the $D_n^{(1)}$-crystal automorphism $\sigma$, induced by the symmetry between the 0 and 1 nodes in the Dynkin diagram.  We give a method for computing $\sigma$, and use it to give expressions for $\tilde{e}_0,\tilde{f}_0,\varepsilon_0,\varphi_0$.

Consider the restriction of $B(k\Lambda_2), 0\leq k \leq s$ from $D_n$ to $D_{n-1}$ by deleting the 1 arrows.   
The $D_{n-1}$ branching components are the irreducible components of the restricted crystal.  These may be partially ordered by setting $B<C$ if there exists a 1-arrow connecting some element of $B$ to some element of $C$ in $B(k\Lambda_2)$.  There is also a rank function relative to $s$ and $k$ associated with each branching component $B$  given by setting $\text{rk}(B)=\#\{C|B<C \text{ in some path from }B \text{ to the highest branching component in }B(k\Lambda_2)\}+s-k+1$.  We define the function $\iota_j^k:B(j\Lambda_2)\cup \{0\}\to B(k\Lambda_2)\cup \{0\},0\leq j,k\leq s$ by defining $\iota_j^k(T)$ to be the element corresponding to $T$ in the corresponding branching component of $B(k\Lambda_2)$ whose rank is equal to that of $B$, if it exists, and 0 if no corresponding branching component exists, and $\iota_k^j(0)=0$.  Then, $\iota_k^{k+1}(T),0\leq k\leq s-1$ may be given explicitly.  Let $T'$ be the $2\times k+1$ tableau defined by the following: if $k\neq 0$, then
\begin{eqnarray*}
T'_{1,k+1} &=&\left \{  \begin{array}{ll}2, &\text{if }T_{1,k}= 1\\
						T_{1,k},&\text{otherwise}\end{array}\right .\\
T'_{2,k+1} &=&\left \{  \begin{array}{ll}\overline{2},& \text{if }T_{1,k}= 1\\
						\overline{1}, &\text{otherwise}\end{array}\right .  
\end{eqnarray*}
and, for $2\leq i \leq k,$
\begin{eqnarray*}
T'_{1,i} &=&\left \{  \begin{array}{ll}T_{1,i-1},& \text{if }T_{2,i}\neq \overline{T_{1,i-1}}\\
						\overline{n}, &\text{if }
			T_{2,i-1}=n,T_{2,i}=\overline{T_{1,i-1}}=\overline{n-1}\\
			T_{1,i-1}+1,&\text{ otherwise}\end{array}\right .\\
T'_{2,i} &=&\left \{  \begin{array}{ll}T_{2,i},& \text{if }T_{2,i}\neq \overline{T_{1,i-1}}\\
						n,& \text{if }T_{2,i-1}=n,T_{2,i}=\overline{T_{1,i-1}}=\overline{n-1}\\
			\overline{T_{2,i}+1},&\text{ otherwise}\end{array}\right .  
\end{eqnarray*}
and,
\begin{eqnarray*}
T'_{1,1} &=&\left \{  \begin{array}{ll}2,& \text{if }T_{2,1}= \overline{1}\\
						1, &\text{otherwise}\end{array}\right .\\
T'_{2,1} &=&\left \{  \begin{array}{ll}\overline{2},& \text{if }T_{2,1}=\overline{1}\\
						T_{2,1}, &\text{otherwise}\end{array}\right .  
\end{eqnarray*}
finally,
\begin{equation*}
T'=\begin{array}{|c|}\hline 2\\ \hline \overline{2} \\ \hline\end{array}\:, \text{if }k=0.
\end{equation*}
Then $\iota_k^{k+1}(T)=T'$.  Conversely, let $T'$ be the inverse image of $T\in B((k+1)\Lambda_2)$, if it exists.  If so, and if $T'\in B(k\Lambda_2),$ then $\iota_{k+1}^{k}(T)=T'$, otherwise $\iota_{k+1}^{k}(T)=0.$  Finally, $$\iota_{j}^k(T)=\left \{ \begin{array}{ll}
	\iota_{k-1}^k\circ \cdots\circ \iota_{j+1}^{j+2} \circ \iota_{j}^{j+1}(T),&\text{if }1\leq j<k\leq s\\
	\iota_{k+1}^k\circ \cdots\circ \iota_{j-1}^{j-2} \circ \iota_{j}^{j-1}(T),&\text{if }1\leq k<j\leq s\\
	T, &\text{if }1\leq j=k\leq s.
	\end{array}
\right .$$

The partially ordered set of branching components of $B(k\Lambda_2)$ also has a symmetry with respect to the rank function relative to $s$, called $*\mathcal{BC}$-duality.  If $B$ is a branching component of $B(k\Lambda_2)$ with $\text{rk}(B)=r,$ then there is an isomorphic branching component $B^{*\mathcal{BC}}$ with $\text{rk}(B^{*\mathcal{BC}})=2s-r$ which is the image as a set of $B$ under the Lusztig automorphism \cite{Lu2}. For $T\in B\subset B(k\Lambda_2),$ we define $T^{*\mathcal{BC}}$ to be the element corresponding to $T$ in $B^{*\mathcal{BC}}$.  Then $*\mathcal{BC}$ commutes with $\iota_k^j$ for all $j$ for which $\iota_k^j(T)\neq 0$.  Using this fact, we may compute $T^{*\mathcal{BC}}$ explicitly as follows.  Let $l=\min\{j|\iota_k^j(T)\neq 0\}$.  We define a map $\psi$ in $B(l\Lambda_2)$ as follows.  For $T\in B(l\Lambda)$, let $a=\#\{i|T_{1,i}=1\}, b=\#\{i|T_{2,i}=\overline{1}\}.$  Iteratively compute the following: for $1\leq i \leq a-1$ let $T_{1,i}'=T_{1,i}, T_{2,i}'=T_{2,i}$, and let $T_{2,a}'=T_{2,a}.$  Starting from $i=a$ until we reach $i=m$ such that $T_{1,m+1}\geq T_{2,m}'$ and $(T_{1,m+1}, T_{2,m}')\neq (n,\overline{n}),(\overline{n},n)$, and $(T_{1,m+1},T_{2,m+1},T_{2,m}')\neq (x,\overline{x},y)$ with $x\leq y\leq \overline{x},$ if such an index, $m$, exists, or $l-b$ otherwise:
\begin{eqnarray*}
T_{1,i}'&=&\left \{\begin{array}{ll}
	T_{1,i+1}, &\text{if }T_{2,i+1}\neq \overline{T_{1,i+1}}\\
	n-1,& \text{if }T_{2,i}'=n,T_{1,i+1}=\overline{T_{2,i+1}}=\overline{n}\\
	T_{1,i+1}-1,&\text{otherwise},
\end{array} \right .\\
T_{2,i+1}'&=&\left \{\begin{array}{ll}
	T_{2,i+1}, &\text{if }T_{2,i+1}\neq \overline{T_{1,i+1}}\\
	\overline{n-1},& \text{if }T_{2,i}'=n,T_{1,i+1}=\overline{T_{2,i+1}}=\overline{n}\\
	\overline{T_{2,i+1}-1},&\text{otherwise},
\end{array} \right \}, i\neq m
\end{eqnarray*}
and let
\begin{equation*}
x=\left \{\begin{array}{ll}
	T_{2,m}, &\text{if }T_{2,m}\neq \overline{T_{1,m}}\\
	\overline{n-1},& \text{if }T_{2,m-1}=n,T_{1,m}=\overline{T_{2,m}}=\overline{n}\\
	\overline{T_{2,m}-1},&\text{otherwise}.
\end{array} \right .
\end{equation*}
Then, let
\begin{eqnarray*}
T_{1,m}'&=&\left \{\begin{array}{ll}x, &\text{if }x\neq \overline{T_{2,m+1}}, \text{ or } m=l\\
			\overline{n},& \text{if }T_{1,m+1}=\overline{n},T_{2,m+1}=\overline{x}=\overline{n-1}\\
			x+1, &\text{otherwise},
\end{array}\right .
\end{eqnarray*}
and,
\begin{eqnarray*}
T_{2,m}'&=&\left \{\begin{array}{ll}T_{2,m+1}, &\text{if }x\neq \overline{T_{2,m+1}}\\
			n,& \text{if }T_{1,m+1}=\overline{n},T_{2,m+1}=\overline{x}=\overline{n-1}\\
			\overline{T_{2,m+1}+1}, &\text{otherwise},
\end{array}\right \}, m\neq l
\end{eqnarray*}
and, for $m+1\leq i\leq l-b-1:$
\begin{eqnarray*}
T_{1,i}'&=&\left \{\begin{array}{ll}
	T_{1,i}, &\text{if }T_{1,i}\neq \overline{T_{2,i+1}}\\
	\overline{n},& \text{if }T_{1,i+1}=\overline{n},T_{2,i+1}=\overline{T_{1,i}}=\overline{n-1}\\
	T_{1,i}+1,&\text{otherwise},
\end{array} \right .\\
T_{2,i}'&=&\left \{\begin{array}{ll}
	T_{2,i+1}, &\text{if }T_{1,i}\neq \overline{T_{2,i+1}}\\
	n,& \text{if }T_{1,i+1}=\overline{n},T_{2,i+1}=\overline{T_{1,i}'}=\overline{n-1}\\
	\overline{T_{2,i+1}+1},&\text{otherwise},
\end{array} \right .
\end{eqnarray*}
Finally, let $T'_{2,l-b}=\overline{1},$ and $T'_{i,j}=T_{i,j}$ in all other cases.  Defining $\psi(T)=T'$, we have: $T^{*\mathcal{BC}}=\iota_{l}^k\psi^{a-b}(\iota_{k}^l (T))$.

Finally, for $T\in B(k\Lambda_2)\subset B^{2,s}$ we have the following:
\begin{eqnarray*}
\sigma(T)&=&\iota_{k}^{s+l-k}(T^{*\mathcal{BC}})\\
\tilde{e}_0T&=&\sigma(\tilde{e}_1\sigma(T))\\
\tilde{f}_0T&=&\sigma(\tilde{f}_1\sigma(T))\\
\varepsilon_0(T)&=&\varepsilon_1(\sigma(T))\\
\varphi_0(T)&=&\varphi_1(\sigma(T)),
\end{eqnarray*}
where $l=\min\{j|\iota_{k}^j(T^{*\mathcal{BC}})\neq 0\}.$

\emph{Example:} Let $T=\begin{array}{|c|c|}\hline 1&2\\\hline 2&\overline{2}\\ \hline\end{array}$ be an element of the $D_4^{(1)}$-crystal $B^{2,2}.$  $T^{*\mathcal{BC}}=$.  We compute $\tilde{e}_0(T)$.  The minimum $l\in \{ 0,1,2\}$ such that $\iota_2^l(T)\neq 0$ is $1$, with $\iota_{2}^1(T)=\begin{array}{|c|}\hline 1\\\hline 2\\ \hline\end{array}\:$.  Then $\psi\left (\:\begin{array}{|c|}\hline 1\\\hline 2\\ \hline\end{array}\: \right )=\begin{array}{|c|}\hline 2\\\hline \overline{1}\\ \hline\end{array}\:$, and we have $\sigma(T)=\iota_1^2T^{*\mathcal{BC}}=\iota_2^1\circ\iota_1^2\left (\:\begin{array}{|c|}\hline 2\\\hline \overline{1}\\ \hline\end{array}\:\right )=\begin{array}{|c|}\hline 2\\\hline \overline{1}\\ \hline\end{array}\:$.  Then $\tilde{e}_1(\sigma(T))=\begin{array}{|c|}\hline 2\\\hline \overline{2}\\ \hline\end{array}\:.$  On this tableau, we have $(\tilde{e}_1(\sigma(T)))^{*\mathcal{BC}}=\begin{array}{|c|}\hline 2\\\hline \overline{2}\\ \hline\end{array}\:,$ and the minimum $l'\in \{0,1,2\}$ such that $\iota_1^{l'}((\tilde{e}_1(\sigma(T)))^{*\mathcal{BC}})\neq 0$ is $0$, so we have $\tilde{e}_0(T)=\sigma(\tilde{e}_1\sigma(T))=\begin{array}{|c|}\hline 2\\\hline \overline{2}\\ \hline\end{array}\:.$
\subsection{Lecouvey's insertion algorithm for $D_n$-crystals}
Let $b_1 \otimes b_2 \otimes b_3 \otimes \cdots\otimes b_l$ be a highest weight element of the $D_n$-crystal $B(\Lambda_1)^{\otimes l}$ and $B(b_1\; b_2 \; \dots \; b_l)$ be the $D_n$-crystal generated by $b_1 \otimes b_2 \otimes b_3 \otimes \cdots \otimes b_l$.  We review the ``column insertion'' algorithm in \cite{L}.  Let $\xi:B(1 \; 2\; 1)\to B(1 \; 1\; 2),$ be the unique isomorphism between the given crystals.  Explicitly:
\begin{equation*}
\xi(x\otimes y \otimes z)=\left \{\begin{array}{ll}
x\otimes z\otimes y, &\text{if }z\leq x <y, y\neq \overline{z},\\
y\otimes x\otimes z, &\text{if }x< z \leq y, y\neq \overline{x},\\
x\otimes(z+1)\otimes \overline{z+1}, &\text{if }y=\overline{z}, z<n-1,z< x < \overline{z},\\
\overline{x-1}\otimes (x-1)\otimes z, &\text{if }y=\overline{x}, 1<x<n, x\leq z \leq \overline{x},\\
y\otimes x\otimes z, &\text{if }y\geq \overline{n-1}, (x,z)=(\overline{n},n)\text{ or }(n,\overline{n})\\ 
x\otimes z\otimes y, &\text{if }z\leq n-1, (x,y)=(\overline{n},n)\text{ or }(n,\overline{n})\\ 
\overline{n-1}\otimes (n-1)\otimes z,& \text{if }(x,y,z)=(n,\overline{n},\overline{n})\text{ or }(\overline{n},n,n)\\
\overline{n}\otimes \overline{n}\otimes n,& \text{if }(x,y,z)=(\overline{n},\overline{n-1},n-1),\\
n\otimes n\otimes \overline{n},& \text{if }(x,y,z)=(n,\overline{n-1},n-1).\\
\end{array}
\right .
\end{equation*}
Let 
$$\begin{array}{l}
\xi_i:B(1)^{\otimes i-1}\otimes B(1\; 2 \; 1)\otimes B(1)^{\otimes l-i-2}\to B(1)^{\otimes i-1}\otimes B(1\; 1 \; 2)\otimes B(1)^{\otimes l-i-2}\\
\end{array}$$
 be the following crystal isomorphism:
$$\begin{array}{l}
\xi_i=1^{\otimes i-1}\otimes \xi \otimes 1^{\otimes l-i-2}.
\end{array}$$
We use these maps to define the insertion of a letter $b\in \mathcal{B}$ into a column $T$, notated $b\to T$, of a Kashiwara-Nakashima tableau (\cite{KN}).  Let $T=T_{1,1}\otimes T_{2,1}\otimes \cdots \otimes T_{k,1},$ where $T_{i,1}\in \mathcal{B}, 1\leq i \leq l$ be the reading of $T$.
\begin{enumerate}
\item If $T=\varnothing$ then $b\to T=\begin{array}{|c|}\hline b \\ \hline \end{array}\:,$
\item if $T=\begin{array}{|c|}\hline T_{1,1} \\ \hline \end{array}\:$, in other words if it consists of a single box and $b\leq T_{1,1}$, then $b\to \begin{array}{|c|}\hline T_{1,1} \\ \hline \end{array}=\begin{array}{|c|c|}\hline b &T_{1,1} \\ \hline \end{array}\:,$
\item if $b> T_{k,1}$ or $(T_{k,1},b)=(n,\overline{n}),$ or $(\overline{n},n)$, and there exists some $y\leq n\in \mathcal{B}$ in the sequence $S=(T_{1,1},T_{2,1},\dots, T_{k,1},b)$ such that $\overline{y}$ is also in $S$ and $\#\{x\in S| x\leq y \text{ or }x\geq \overline{y}\}>y$, then $b\to T$ is the column formed by removing $(z,\overline{z})$ where $z$ is the least such letter that occurs,
\item otherwise, if $b> T_{k,1}$ or $(T_{k,1},b)=(n,\overline{n}),$ or $(\overline{n},n)$ then $b\to T$ is just the column $T$ with $\begin{array}{|c|}\hline b\\ \hline \end{array}$ appended to the bottom,
\item otherwise, it is the case that $b < T_{1,k},k\geq 2$.  In this case, let $T_{1,2}'\otimes T_{1,1}'\otimes T_{2,1}'\cdots \otimes T_{1,k}'=\xi_{1}\xi_{2}\cdots\xi_{k-1}(T_{1,1}\otimes T_{2,1}\otimes \cdots \otimes T_{k,1}\otimes b)$.  Then:
\begin{equation*}
b\to T=\begin{array}{@{\vline}c@{\vline}c|}\hline T_{1,1}'&\:T_{1,2}'\\ \hline \begin{array}{@{\:}c@{\:}}T_{2,1}' \\ \hline \vdots \\ \hline T_{k,1}' \\ \hline\end{array}
\end{array}
\end{equation*}
\end{enumerate}
Now, let $T=T^1 T^2\cdots  T^k$, where $T^1,T^2,\dots, T^k$ are the columns of $T$.  Define the insertion $b\to T,b\in \mathcal{B}$ recursively as follows:
\begin{enumerate}
\item Compute $b\to T^1$ as above.  If case 4 occurs, then 
\begin{equation*}
b\to T=(b\to T^1) T^2 \cdots  T^k.
\end{equation*}
\item If case 3 occurs, let $b\to T=T_{1,1}'\otimes T_{1,2}'\otimes \cdots \otimes T_{1,l}'$ for some $l\geq 0$.  Then $b\to T$ is given by successively inserting the letters of $(T^1)'$ into $T\backslash T^1$as follows:
\begin{equation*}
b\to T=T_{l,1}'\to (T_{l-1,1}'\to(\cdots \to T_{1,1}'\to  T^2 T^3 \cdots T^k)).
\end{equation*}
\item If case 2 or 5 occurs, then notice that in both cases, $b\to T^1= (T^1)' (T^2)'$ where $(T^1)'$ is a column and $(T^2)'=\begin{array}{|c|}\hline c\\ \hline \end{array}\:,$ for some $c\in \mathcal{B}$.  Then we define:
\begin{equation*}
b\to T=(T^1)'(b\to T^2 T^3 \cdots T^k).
\end{equation*}
\end{enumerate}
Lecouvey \cite{L} defines an oscillating tableau $Q$ of type $D_n$ to be a sequence $Q_0,Q_1,Q_2,\dots,Q_l$ of pairs $Q_k=(O_k,\varepsilon_k)$, where $O_k$ is a Young diagram whose columns have height $\leq n$ and $\varepsilon_k\in \{-,0,+\}$, satisfying, for $k=1,2,\dots, l$:
\begin{enumerate}
\item $O_k$ and $O_{k+1}$ differ by exactly one box,
\item $\varepsilon_k \neq 0$ and $\varepsilon_{k+1}\neq 0$ imply $\varepsilon_k=\varepsilon_{k+1}$, and
\item $\varepsilon_k=0$ if and only if $O_k$ has no columns of height $n$.
\end{enumerate}
Then, we have the following analogue of the Robinson-Schensted correspondence for type $D_n$:
\begin{theorem}[\cite{L}]\label{insertion}
There is a bijection between the set $B(\Lambda_1)^{\otimes l}$ and the set of all pairs $(P,Q)$  where $Q=(Q_0=(\varnothing,0),Q_1,Q_2,\dots, Q_l)$ is an oscillating tableau, and $P$ is a Kashiwara-Nakashima tableau of shape $O_l$ such that, if $P$ has a column of height $n$, whose $k$th entry is $n$ (resp. $\overline{n}$) then $n-k$ is even (resp. odd) if $\varepsilon_l=+$, and $n-k$ odd (resp. even) if $\varepsilon_l=-$.  
\end{theorem}
Explicitly, this bijection is given by sending $b_1\otimes b_2\otimes \cdots \otimes b_l\in B(\Lambda_1)^{\otimes l}$ to $(P_l,Q)$, where $P_k=b_k \to(b_{k-1}\to(\cdots b_1\to\varnothing)), 1\leq k \leq l$, $Q_k=(O_k,\varepsilon_k),1\leq k \leq l$, $O_k$ the shape of the tableau $P_k$, and $\varepsilon_k=0$ is $P_k$ has no columns of height $n$, $\varepsilon_k=+$ if $P_k$ has a column of height $n$ whose $k$th entry is $n$ (resp. $\overline{n}$) such that $n-k$ is even (resp. odd), and $\varepsilon_k=-$ otherwise.

\emph{Example:} The element $2\otimes 4 \otimes \overline{4} \otimes 3$ of the $D_4$-crystal $B(\Lambda_1)^{\otimes 4}$ corresponds to the following sequence of tableaux:
\begin{itemize}
\item $2\ \to \varnothing =\begin{array}{|c|}\hline 2\\ \hline \end{array}\:$,
\item $4\to \begin{array}{|c|}\hline 2\\ \hline \end{array}=\begin{array}{|c|}\hline 2\\ \hline 4\\ \hline \end{array}\:,$
\item $\overline{4}\to \begin{array}{|c|}\hline 2\\ \hline 4 \\ \hline\end{array}=\begin{array}{|c|}\hline 2\\ \hline 4\\ \hline \overline{4}\\ \hline\end{array}\:,$
\item $3 \to \begin{array}{|c|}\hline 2\\ \hline 4\\ \hline \overline{4}\\ \hline\end{array}\:$.  We compute $\xi_1\xi_2(2\otimes 4\otimes \overline{4}\otimes 3)=\xi_1(2\otimes 4\otimes 3\otimes \overline{4})=4\otimes 2\otimes 3\otimes 4=T_{1,2}'\otimes T_{1,1}'\otimes T_{2,1}'\otimes T_{3,1}'.$ The resulting tableau is: $\begin{array}{@{\vline}c@{\vline}c|}\hline 2&\:4\\ \hline \begin{array}{@{\:}c@{\:}}3 \\ \hline \overline{4}\\ \hline\end{array}\end{array}\:.$
\end{itemize}
\subsection{Combinatorial $R:B^{r,s}\otimes B^{r',s'}\to B^{r',s'}\otimes B^{r,s}$}

In this section, we define the combinatorial $R$ matrix for $B^{r,s}\otimes B^{r',s'}$ and the energy function $H:B^{r,s}\otimes B^{r',s'}\to \mathbb{Z}$.

\begin{proposition}[\cite{KMN}]
There exist a unique crystal isomorphism $\mathcal{R}:B^{r,s}\otimes B^{r',s'}\to B^{r',s'}\otimes B^{r,s}$ and a function $H:B^{r,s}\otimes B^{r',s'}\to \mathbb{Z}$ unique up to an additive constant satisfying the following property: for any $b\in B^{r,s}$, $b'\in B^{r',s'},$ and $i\in I$ such that $\tilde{e_i}(b\otimes b')\neq 0$, 
\begin{equation}
H(\tilde{e}_i(b\otimes b'))=\left \{ \begin{array}{ll}
		H(b\otimes b')+1 & \text{if }i=0, \varphi_0(b)\geq \varepsilon_0(b'), \varphi_0(\tilde{b}')\geq \varepsilon_0(\tilde{b}),\\
		H(b\otimes b')-1 & \text{if }i=0, \varphi_0(b) < \varepsilon_0(b'), \varphi_0(\tilde{b}') < \varepsilon_0(\tilde{b}),\\
		H(b\otimes b') & \text{otherwise,}
\end{array}
\right .
\end{equation}
where $\tilde{b}'\otimes \tilde{b}=\mathcal{R}(b\otimes b').$  $H$ is called an \emph{energy function} on $B^{r,s}\otimes B^{r',s'}.$ 
\end{proposition}

\begin{proposition}[Yang-Baxter Equation]
The following equation holds on $B^{r,s}\otimes B^{r',s'}\otimes B^{r'',s''}$
\begin{equation}
(\mathcal{R}\otimes 1)(1\otimes \mathcal{R})(\mathcal{R}\otimes 1)=(1\otimes \mathcal{R})(\mathcal{R}\otimes 1)(1\otimes \mathcal{R}), \label{YBE}
\end{equation}
where $1$ denotes the identity map.
\end{proposition}
We  define $\mathcal{R}_{i,i+1}$ to be the map 
$$1^{\otimes i-1}\otimes \mathcal{R}\otimes 1^{\otimes L-i-1}:\bigotimes_{k=1}^L B^{r_k,s_k}\to \bigotimes_{k=1}^L B^{r_{\tau(k)},s_{\tau(k)}}$$
where $\tau$ transposes $i,i+1$ and fixes all other integers.  In this notation, \eqref{YBE} becomes:
\begin{equation}
\mathcal{R}_{i,i+1}\mathcal{R}_{i+1,i+2}\mathcal{R}_{i,i+1}=\mathcal{R}_{i+1,i+2}\mathcal{R}_{i,i+1}\mathcal{R}_{i+1,i+2}
\end{equation}
for $L\geq 3, i=1,2,\dots, L-2.$

\begin{definition}[Affinization]
The \emph{affinization} of $B^{r,s}$ is defined to be the set
$$
\{ z^nb | n\in \mathbb{Z}, b\in B^{r,s}\}
$$
where the action of $\tilde{e}_i, \tilde{f}_i$ is defined as:
\begin{eqnarray*}
\tilde{e}_i(z^nb)&=&z^{n+\delta_{i0}}\tilde{e}_i(b)\\
\tilde{f}_i(z^nb)&=&z^{n-\delta_{i0}}\tilde{f}_i(b).
\end{eqnarray*}
\end{definition}
The combinatorial $R$-matrix $\mathcal{R}^{\text{Aff}}:\text{Aff}(B^{r,s})\otimes \text{Aff}(B^{r',s'})\to\text{Aff}(B^{r',s'})\otimes \text{Aff}(B^{r,s})$ is given by:
\begin{equation}
\mathcal{R}^{\text{Aff}}(z^m b\otimes z^n b')=z^{n+H(b\otimes b')}\tilde{b}' \otimes z^{m-H(b\otimes b')}\tilde{b}
\end{equation}
where $\tilde{b}'\otimes \tilde{b}=\mathcal{R}(b\otimes b')$.
\subsection{Combinatorial $R$-matrix for $B^{1,s}\otimes B^{1,s'}$}
The combinatorial $R$-matrix for the $D_n^{(1)}$-crystals $B^{1,s}\otimes B^{1,s'}$ has been given in \cite{HKOT}.  For the tableaux:

$$
b\otimes b' = (x_1,x_2,\dots,\overline{x}_1)\otimes (y_1,y_2,\dots, \overline{y}_1)\in B^{1,s}\otimes B^{1,s'}
$$
 we set $z=\min(x_1,\overline{y}_1)$ so that $b\otimes b'$ are associated with the tableaux:
$$
\begin{array}{l @{} l @{} l @{} l }\underbrace{\begin{array}{|c|c|c|c|}\hline 1 & 1 & \cdots & 1 \\ \hline\end{array}}_{z}& 
\begin{array}{c|} \hline T_* \\ \hline \end{array}\:\otimes\: 
\begin{array}{|c|c|c|c|}\hline v_1 & v_2 & \cdots & v_k \\ \hline\end{array}
&\underbrace{\begin{array}{c|c|c|c|}\hline  \overline{1} & \overline{1} & \cdots & \overline{1}\\ \hline \end{array}}_{z}\end{array}
$$
where $k=s'-z$.  Now, define the tableau $T^{(0)}=v_1\to(v_2\to (\cdots v_k\to T_*)\cdots)$, which is necessarily of the form, for $l=s-z$:
$$
\begin{array}{l @{} l @{} l @{} }
\begin{array}[t]{|c}\hline j_1 \\ \hline \end{array}&\begin{array}[t]{|c|c|} \hline \cdots \cdots \cdots& j_k \\ \hline \end{array}&
\begin{array}[t]{c|c|c|}\hline i_{m+1}&\cdots& i_{l}\\ \hline \end{array}\\
\begin{array}[t]{|c} i_1 \hspace{.3mm}\\ \hline \end{array}&\begin{array}[t]{|c|c|} \cdots & i_m\\ \hline \end{array}
\end{array}
$$
for some $0\leq m \leq \lfloor\frac{k+l}{2} \rfloor$.  Next, we use reverse column insertion to remove the boxes containing $i_{l},i_{l-1},\dots,i_{m+1}, i_m,\cdots, i_1$ in order.  Let $w_1,w_2,\dots, w_{l}$ be the sequence of letters that are produced in each step, and $T^{(1)},T^{(2)},\dots T^{(k)}$ denote the sequence of tableaux.  Then 
\begin{equation*}
\mathcal{R}(b\otimes b')=\begin{array}{l @{} l @{} l @{} l }\underbrace{\begin{array}{|c|c|c|c|}\hline 1 & 1 & \cdots & 1 \\ \hline\end{array}}_{z}& 
\begin{array}{c|} \hline T^{(l)} \\ \hline \end{array} \otimes 
\begin{array}{|c|c|c|c|}\hline w_1 & w_2 & \cdots & w_l \\ \hline\end{array}& 
\underbrace{\begin{array}{c|c|c|c|}\hline  \overline{1} & \overline{1} & \cdots & \overline{1}\\ \hline \end{array}}_{z}\end{array}
\end{equation*}
The energy function is given by:
\begin{equation*}
H(b \otimes b')=2\min(l,k)-m-2s_2.
\end{equation*}
\subsection{Combinatorial $R$-matrix for $A_n^{(1)}$}
Define $\widehat{B}^{r,s}$ to be the set of $r\times s$ semistandard tableaux in the alphabet $\{1',2',\dots, n'\}.$  The crystal structure of this set is given in \cite{Sh}, and the combinatorial $R$ matrix is given as well in loc. cit. (see also \cite{Yd1}) by the following row insertion procedure.

To an element $$T=\begin{array}{|c|c|c|c|}\hline T_{1,1}&T_{1,2}&\cdots &T_{1,s}\\ \hline T_{2,1}&T_{2,2}&\cdots & T_{2,s}\\ \hline \vdots & \vdots & \ddots & \vdots\\ 
\hline T_{r,1}&T_{r,2}&\cdots &T_{r,s}\\ \hline \end{array}\in \widehat{B}^{r,s}$$
we associate the row word $\text{row}(T)=T_{1,s}\otimes T_{1,s-1}\otimes T_{1,1}\otimes T_{2,s}\otimes T_{2,s-1}\otimes \cdots \otimes T_{2,1}\otimes \cdots \otimes T_{r,s}\otimes T_{r,s-1}\otimes \cdots \otimes T_{r,1}.$   Then we have the following, which follows from the fact that the decomposition of $\widehat{B}^{r,s}\otimes \widehat{B}^{r',s'}$ is outer-multiplicity free as $A_n$-crystals:
\begin{theorem}[\cite{Sh}, see \cite{Yd1}]\begin{enumerate}Let $T\otimes T'\in \widehat{B}^{r,s}\otimes \widehat{B}^{r',s'}$.  Then:
\item the map $\widehat{\mathcal{R}}:\widehat{B}^{r,s}\otimes \widehat{B}^{r',s'}\to \widehat{B}^{r',s'}\otimes \widehat{B}^{r,s}$ is given by the following: 
$$
\widehat{\mathcal{R}}(T\otimes T')=\widetilde{T'}\otimes \widetilde{T}
$$
if and only if $\text{row}(T')\to T=\text{row}(\widetilde{T})\to T',$ where $\text{row}(T)\to T'$ denotes the resulting tableau from row inserting all the letters of $\text{row}(T)$ in order into $T'$,
\item let $d(T,T')$ denote the number of nodes in the shape of $\text{row}(T)\to T'$ that are strictly to the right of the $\max(s,s')$ column.  Then the energy function $\widehat{H}:B^{r,s}\otimes B^{r',s'}\to \mathbb{Z}$ is given by
$$
\widehat{H}(T\otimes T')=d(T,T')-\min(r,r')\min(s,s').
$$
\end{enumerate}
\end{theorem}
To the tableau $T\in \widehat{B}^{r,s}$ we associate the coordinatization $$\begin{pmatrix}x_{1,1}&x_{1,2}&\cdots&x_{1,n}\\
x_{2,1}&x_{2,2}&\cdots & x_{2,n}\\
\vdots&\vdots&\ddots& \vdots\\
x_{r,1}& x_{r,2}&\cdots &x_{r,n}
\end{pmatrix}
\in \mathbb{Z}_{\geq 0}^{r\times n}.$$
\section{Combinatorial $R$-matrix for $D_n^{(1)}$}
In this section, we give the combinatorial $R$-matrix for the $U_q(D_n^{(1)})$-crystals $B^{2,s}\otimes B^{2,1}$, first for the highest weight vectors and then for arbitrary vectors using the analogue of the Robinson-Shensted correspondence (Theorem \ref{insertion}).

The $D_n,n\geq4$ highest weight vectors of $B^{2,s}\otimes B^{2,1}$ are as follows:
\begin{center}
\begin{tabular}{|l|l|}
\hline
Highest weight vector & Classical Weight\\
\hline\hline
\rule{0pt}{5ex}$u_k\otimes \begin{array}{|c|}\hline 1\\ \hline 2\\ \hline \end{array}\:,0\leq k \leq s$ & $(k+1)\Lambda_2$\\
$u_k\otimes \varnothing, 0\leq k \leq s$& $k\Lambda_2$\\
\rule{0pt}{5ex}$u_k\otimes \begin{array}{|c|}\hline 1\\ \hline 3\\ \hline \end{array}\:,1\leq k \leq s$ & 
								$\left \{\begin{array}{l}\Lambda_1+(k-1)\Lambda_2+\Lambda_3,n>4\\
								\Lambda_1+(k-1)\Lambda_2+\Lambda_3+\Lambda_4,n=4\end{array}\right .$\\
$u_k\otimes \begin{array}{|c|}\hline 3\\ \hline 4\\ \hline \end{array}\:, 1\leq k \leq s$ & $\left \{\begin{array}{l}(k-1)\Lambda_2+\Lambda_4,n>5\\
										         (k-1)\Lambda_2+\Lambda_4+\Lambda_5,n=5\\
											(k-1)\Lambda_2+2\Lambda_4,n=4\end{array}\right .$\\
$u_k\otimes \begin{array}{|c|}\hline 1\\ \hline \overline{2}\\ \hline \end{array}\:,1\leq k \leq s$ & $2\Lambda_1+(k-1)\Lambda_2$\\
\rule{0pt}{5ex}$u_k\otimes \begin{array}{|c|}\hline 3\\\hline \overline{2}\\ \hline \end{array}\:,2\leq k \leq s$ & $\left \{\begin{array}{l}
									\Lambda_1+(k-2)\Lambda_2+\Lambda_3,n>4\\
									\Lambda_1+(k-2)\Lambda_2+\Lambda_3+\Lambda_4,n=4
									\end{array}\right .$\\
\rule{0pt}{5ex}$u_k\otimes \begin{array}{|c|}\hline 3\\\hline \overline{3}\\ \hline \end{array}\:, 1\leq k \leq s $& $k\Lambda_2$\\
\rule{0pt}{5ex}$u_k\otimes \begin{array}{|c|}\hline \overline{2}\\ \hline \overline{1}\\ \hline \end{array}\:,1\leq k \leq s$ & $(k-1)\Lambda_2$\\
\rule{0pt}{5ex}$u_k\otimes \begin{array}{|c|}\hline 3\\ \hline \overline{4}\\ \hline \end{array}\:,n=4,1\leq k \leq s$ & $(k-1)\Lambda_2+2\Lambda_3$\\
\hline
\end{tabular}
\end{center}
We will also use the $D_n,n\geq4$ highest weight vectors of $B^{1,1}\otimes B^{2,1}$ and $B^{2,1}\otimes B^{1,1}$:
\begin{center}
\begin{tabular}{|l|l|}
\hline
Highest weight vector in& Classical \\
$B^{1,1}\otimes B^{2,1}$ &weight\\
\hline \hline&\\
$\begin{array}{|c|}\hline 1\\ \hline  \end{array}\:\otimes \:\begin{array}{|c|}\hline 1\\ \hline 2\\ \hline\end{array}$ & $\Lambda_1 + \Lambda_2$\\
&\\
$\begin{array}{|c|}\hline 1\\ \hline  \end{array}\:\otimes \:\begin{array}{|c|}\hline 2\\ \hline 3\\ \hline\end{array}$ & $\Lambda_3+\delta_{n,4}\Lambda_4$\\
&\\
$\begin{array}{|c|}\hline 1\\ \hline  \end{array} \:\otimes \:\varnothing$ & $\Lambda_1$\\
&\\
$\begin{array}{|c|}\hline 1\\ \hline  \end{array}\:\otimes \:\begin{array}{|c|}\hline 2\\ \hline \overline{2}\\ \hline \end{array}$ & $\Lambda_1$\\
&\\
\hline
\end{tabular}
\begin{tabular}{|l|l|}
\hline
Highest weight vector in& Classical \\
$B^{2,1}\otimes B^{1,1}$ &weight\\
\hline \hline
&\\
$\begin{array}{|c|}\hline 1\\ \hline 2\\ \hline\end{array}\: \otimes \:\begin{array}{|c|}\hline 1\\ \hline\end{array}$ & $\Lambda_1 + \Lambda_2$\\
&\\
$\begin{array}{|c|}\hline 1\\ \hline 2\\ \hline\end{array}\: \otimes \:\begin{array}{|c|}\hline 3\\ \hline\end{array}$ & $\Lambda_3+\delta_{n,4}\Lambda_4$\\
&\\
$\varnothing \otimes \:\begin{array}{|c|}\hline 1\\ \hline\end{array}$ & $\Lambda_1$\\
&\\
$\begin{array}{|c|}\hline 1\\ \hline 2 \\ \hline\end{array}\:\otimes\:\begin{array}{|c|}\hline\overline{2}\\ \hline \end{array} $ & $\Lambda_1$\\
&\\
\hline
\end{tabular}
\end{center}

\subsection{Combinatorial $R$ matrix for highest weight elements}
We begin by computing the 0-string through certain elements of $B^{2,s}.$
\begin{lemma}
If $T=\begin{array}{|c|}\hline T_{1,1}\\ \hline T_{2,1} \\ \hline\end{array}\in B(\Lambda_2)\subset B^{2,s}, s\geq 1,$ then the $0$-string through $T$ is given by the following:
\begin{equation}
\tilde{f}_0^jT=\left \{ \begin{array}{ll}
			\begin{array}{|c|c|}\hline 1^{j-1}& 1\\\hline2^{j-1} & T_{1,1}\\ \hline\end{array}\:, & 
						\text{if }T_{2,1}=\overline{2},T_{1,1}\notin \{1,2\},1\leq j\leq s \vspace{1mm}\\  
			\begin{array}{|c|c|}\hline 1^{j-1}&2\\\hline 2^{j-1} &T_{1,1}\\ \hline\end{array}\:, &
						\text{if }T_{2,1}=\overline{1}, T_{1,1}\notin\{2,\overline{2}\},1\leq j\leq s\vspace{1mm} \\
			\begin{array}{|c|}\hline 1^{j-1}\\ \hline 2^{j-1} \\ \hline\end{array}\:, &
			 \text{if } T=\begin{array}{|c|}\hline \overline{2}\\ \hline \overline{1} \\ \hline\end{array},1\leq j \leq s+1\vspace{1mm} \\
			\begin{array}{|c|c|}\hline 1^j&T_{1,1}\\\hline 2^j& T_{2,1}\\ \hline\end{array}\:, & 
						\text{otherwise, if } 1\leq j\leq s-1\\ 
			0,& \text{otherwise.}
\end{array}\right .
\end{equation}
Similarly, we have:
\begin{equation}
\tilde{e}_0^jT=\begin{cases}
			\begin{array}{|c|c|}\hline  T_{2,1}&\overline{2}^{j-1}\\\hline \overline{1}&\overline{1}^{j-1}\\ 
			\hline\end{array}\:, &\text{if }T_{1,1}=2,T_{2,1}\neq \overline{2},1\leq j\leq s\vspace{1mm} \\
			\begin{array}{|c|c|}\hline  T_{2,1}&\overline{2}^{j-1}\\\hline \overline{2}&\overline{1}^{j-1}\\
			\hline\end{array}\:, & \text{if }T_{1,1}=1,T_{2,1}\neq 2,1\leq j\leq s\vspace{1mm} \\
			\begin{array}{|c|}\hline \overline{2}^{j-1}\\ \hline \overline{1}^{j-1} \\ \hline\end{array}\:, &
			\text{if } T=\begin{array}{|c|}\hline 1\\ \hline 2 \\ \hline\end{array}\:, 1\leq j\leq s+1\vspace{1mm}\\ 
			\begin{array}{|c|c|}\hline T_{1,1}&\overline{2}^j\\\hline  T_{2,1}&\overline{1}^j\\ \hline\end{array}\:, & 		
			\text{if }T_{2,1}\notin \{1,2\},	1\leq j\leq s-1\\
			0,& \text{otherwise.}
\end{cases}
\end{equation}
Also, the 0-string through $T=\begin{array}{|c|c|}\hline 1&\overline{2}\\ \hline  3&\overline{1}\\ \hline\end{array}\:$ is given by:
\begin{eqnarray*}
\tilde{f}_0^jT&=&\begin{array}{|c|c|c|}\hline1^{j-1}&1&3\\ \hline  2^{j-1}&3&\overline{3}\\ \hline\end{array}\:, 1\leq j \leq s-1\\
\tilde{e}_0^jT&=&\begin{array}{|c|c|}\hline1&\overline{2}^{j+1}\\ \hline  3&\overline{1}^{j+1}\\ \hline\end{array}\:, 1\leq j \leq s-2
\end{eqnarray*}
\end{lemma}
\begin{proof}
For $T=\begin{array}{|c|}\hline T_{1,1}\\ \hline T_{2,1} \\ \hline\end{array}\in \mathcal{B}(\Lambda_2)\subset B^{2,s}, s\geq 1,$\vspace{-.1in} we compute $\tilde{f}_0^jT=(\sigma\circ \tilde{f}_1 \circ\sigma)^j(T)=\sigma(\tilde{f}_1^j \sigma(T))$, where we have used the fact that $\sigma=\sigma^{-1}.$  In all cases except $T=\begin{array}{|c|}\hline 2\\\hline \overline{2} \\ \hline\end{array}$ we have $l=\min\{j|\iota_1^j(T)\neq 0\}=1$ (we have $\iota_{1}^0\left (\begin{array}{|c|}\hline 2\\ \hline \overline{2} \\ \hline\end{array}\:\right )=\varnothing$, so $l=0$ in that case).
For all $T\in B(\Lambda_2),$ the following may easily be verified (recall that $\begin{array}{|c|}\hline 1 \\ \hline \overline{1}\\ \hline \end{array}\notin B(\Lambda_2)$):
\begin{equation*}
T'=T^{*\mathcal{BC}}=\left \{\begin{array}{ll}\begin{array}{|c|}\hline T_{2,1} \\ \hline \overline{1}\\ \hline \end{array}\:, & 
			\text{if }T_{1,1}=1,\vspace{1mm} \\
		\begin{array}{|c|}\hline 1 \\ \hline T_{1,1}\\ \hline \end{array}\:, & \text{if }T_{2,1}=\overline{1},\vspace{1mm} \\
		T, &\text{otherwise.}
\end{array} \right .
\end{equation*}
Then, we have:
\begin{equation*}
\sigma(T)=\iota_{1}^{s+l-1}(T')=\left \{\begin{array}{ll}
\mathcal{N}_{s-1}\:\begin{array}{|c|}\hline T_{2,1} \\ \hline
 \overline{1} \\ \hline \end{array}\:, & \text{if }T_{1,1}=1\\ 
\begin{array}{|c|c|c|c|}\hline 1  \\ \hline
 T_{1,1} \\ \hline 
\end{array}\:\mathcal{N}_{s-1}\:, & \text{if }T_{2,1}=\overline{1}\\ 
\mathcal{N}_{s-1}, & \text{if }T=\begin{array}{|c|}\hline 2\\\hline \overline{2} \\ \hline\end{array}\\ 
\begin{array}{|c|c|}\hline 1 \\ 
\hline T_{2,1} \\ \hline \end{array}
\:\mathcal{N}_{s-2}\:\begin{array}{|c|c|}\hline  T_{1,1} \\ 
\hline \overline{1} \\ \hline \end{array}\:, & \text{otherwise, if }s>1\\ 
T, &\text{otherwise.}\end{array}\right .
\end{equation*}
where $\mathcal{N}_k=\begin{array}{|c|c|c|c|c|c|}\hline 1^{\lfloor k/2 \rfloor}& 2^{k\pmod{2}}& 2^{\lfloor k/2 \rfloor}\\ \hline
\overline{2}^{\lfloor k/2\rfloor}&\overline{2}^{k \pmod{2}}&\overline{1}^{\lfloor k/2 \rfloor}\\ \hline
\end{array}$ denotes the null configuration of size $k\geq 0$ (\cite{SS}).
We compute:
\begin{equation*}
\tilde{f}_1^j\sigma(T)=\left \{\begin{array}{ll}
\mathcal{N}_{s-1-j}\:\begin{array}{|c|c|}\hline  2^{j}& T_{2,1} \\ \hline
 	 \overline{1}^{j} & \overline{1} \\ \hline \end{array}\:, & \text{if }T_{1,1}=1,1\leq j \leq s-1\vspace{1mm} \\
\begin{array}{|c|}\hline 1 \\ \hline
 T_{1,1} \\ \hline 
\end{array} \:\mathcal{N}_{s-1-j}\:
\begin{array}{|c|}\hline 2^j \\ \hline
 \overline{1}^j \\ \hline 
\end{array} \:, & \text{if }T_{2,1}=\overline{1},1\leq j \leq s-1\vspace{1mm} \\
\begin{array}{|c|c|}\hline 2& 2^{s-1} \\ \hline
T_{1,1} & \overline{1}^{s-1} \\ \hline 
\end{array} \:, & \text{if }T_{2,1}=\overline{1},T_{1,1}\neq 2,j = s\vspace{1mm} \\
\begin{array}{|c|}\hline 2^s \\ \hline
\overline{1}^s \\ \hline 
\end{array} \:, & \text{if }T_{2,1}=\overline{1},T_{1,1}=\overline{2},j = s+1\\
\mathcal{N}_{s-1-j}\:\begin{array}{|c|}\hline  2^{j} \\ \hline
  \overline{1}^{j} \\ \hline 
\end{array}\:, & \text{if }T=\begin{array}{|c|}\hline 2\\\hline \overline{2} \\ \hline\end{array}\:, 1\leq j \leq s-1\\
\begin{array}{|c|}\hline 1 \\ 
\hline  T_{2,1} \\ \hline \end{array}\:\mathcal{N}_{s-2-j}\:
\begin{array}{|c|c|}\hline 2^{j}& T_{1,1} \\ 
\hline \overline{1}^{j} & \overline{1} \\ \hline \end{array}\:,& \text{otherwise, if }s>1, 1\leq j\leq s-2\vspace{1mm} \\
\begin{array}{|c|c|c|}\hline 2& 2^{s-2}& T_{1,1} \\ 
\hline T_{2,1}& \overline{1}^{s-2} & \overline{1} \\ \hline \end{array}\:,&\text{otherwise, if }s>1, 1\leq j=s-1\vspace{1mm}  \\
\begin{array}{|c|c|}\hline 2^{s-1}& T_{1,1} \\ 
\hline \overline{1}^{s-1} & \overline{1} \\ \hline \end{array}\:,& \text{otherwise, if }s>1, T_{2,1}=\overline{2},1\leq j=s\vspace{1mm} \\
\begin{array}{|c|}\hline T_{1,1} \\ 
\hline \overline{1} \\ \hline \end{array}\:,&\text{otherwise, if }s=j=1, T_{2,1}=\overline{2}\\ 
0,&\text{otherwise}. 
\end{array}\right .
\end{equation*}
Thus we have:
\begin{equation*}
\iota_{s-1+l}^{l'-l+1}\tilde{f}_1^j\sigma(T)=\left \{\begin{array}{ll}
\begin{array}{|c|c|}\hline  2^{j}& T_{2,1} \\ \hline
 	 \overline{1}^{j} & \overline{1} \\ \hline \end{array}\:, & \text{if }T_{1,1}=1,1\leq j \leq s-1\vspace{1mm} \\
\begin{array}{|c|c|}\hline 2& 2^{j-1} \\ \hline
 T_{1,1}&\overline{1}^{j-1} \\ \hline 
\end{array} \:, & \text{if }T_{2,1}=\overline{1},T_{1,1}\neq \overline{2},1\leq j \leq s\vspace{1mm} \\
\begin{array}{|c|c|}\hline  2^{j-1} \\ \hline
 \overline{1}^{j-1} \\ \hline 
\end{array} \:, & \text{if }T_{2,1}=\overline{1},T_{1,1}= \overline{2},1\leq j \leq s+1\vspace{1mm} \\
\begin{array}{|c|c|}\hline1& 2^j \\ \hline
 2&\overline{1}^j \\ \hline 
\end{array} \:, & \text{if }T_{1,1}=2,T_{2,1}=\overline{1},1\leq j \leq s-1\vspace{1mm} \\
\begin{array}{|c|c|}\hline 2& 2^{j} \\ \hline
 \overline{2}&\overline{1}^{j} \\ \hline 
\end{array}\:, & \text{if }T=\begin{array}{|c|}\hline 2\\\hline \overline{2} \\ \hline\end{array}\:, 1\leq j \leq s-1\vspace{1mm} \\
\begin{array}{|c|c|c|}\hline 2& 2^{j-2}& T_{1,1} \\ 
\hline T_{2,1}& \overline{1}^{j-2} & \overline{1} \\ \hline \end{array}\:,&\text{otherwise, if }s>1, 1\leq j\leq s-1\vspace{1mm} \\
\begin{array}{|c|c|}\hline 2^{s-1}& T_{1,1} \\ 
\hline \overline{1}^{s-1} & \overline{1} \\ \hline \end{array}\:,& \text{otherwise, if }s>1, T_{2,1}=\overline{2},1\leq j=s\vspace{1mm} \\
\begin{array}{|c|}\hline T_{1,1} \\ 
\hline \overline{1} \\ \hline \end{array}\:,&\text{otherwise, if }s=j=1, T_{2,1}=\overline{2}\vspace{1mm} \\
0,&\text{otherwise}, 
\end{array}\right .
\end{equation*}
where $l'=\min\{j|\iota_{s-1+l}^j\tilde{f}_1^j\sigma(T)\neq 0\}.$  Applying $*\mathcal{BC}$ on each of the tableaux gives the desired result.  A similar computation gives the result for $\tilde{e}_0.$

For the 0-string through $T=\begin{array}{|c|c|}\hline 1&\overline{2}\\ \hline  3&\overline{1}\\ \hline\end{array}\:$ we have $l=2,$ and compute:
\begin{equation*}
\iota_2^{s}(T^{*\mathcal{BC}})=\begin{array}{|c|}\hline 1\\ \hline  3\\ \hline\end{array}\:\mathcal{N}_{s-2}\:\begin{array}{|c|}\hline \overline{2}\\\hline \overline{1}\\ \hline\end{array}
\end{equation*}
\begin{equation*}
\tilde{f}_1^j\iota_2^{s}(T^{*\mathcal{BC}})=\begin{cases}
\begin{array}{|c|}\hline 1\\ \hline  3\\ \hline\end{array}\:\mathcal{N}_{s-j-2}\:
\begin{array}{|c|c|}\hline 2^j&\overline{2}\\\hline \overline{1}^j&\overline{1}\\ \hline\end{array}\:,&\text{if }1\leq j \leq s-2\vspace{1mm} \\
\begin{array}{|c|c|c|}\hline 2&2^{s-2}&\overline{2}\\ \hline 3& \overline{1}^{s-2}&\overline{1}\\ \hline\end{array}\:,&j = s-1\\
0, &\text{otherwise}
\end{cases}
\end{equation*}
\begin{equation*}
\iota_s^{l'}(\tilde{f}_1^j\iota_2^{s}(T^{*\mathcal{BC}}))=\begin{cases}
\begin{array}{|c|c|c|}\hline 2&2^{j-1}&\overline{2}\\ \hline 3& \overline{1}^{j-1}&\overline{1}\\ \hline\end{array}\:,
	&1\leq j \leq s-1\\
0, &\text{otherwise}
\end{cases}
\end{equation*}
(since $l'=j+1$ in all cases) \\
\begin{equation*}
(\iota_s^{l'}(\tilde{f}_1^j\iota_2^{s}(T^{*\mathcal{BC}})))^{*\mathcal{BC}}=\begin{cases}
\begin{array}{|c|c|c|}\hline 1^{j-1}&1&3\\ \hline 2^{j-1}& 3&\overline{3}\\ \hline\end{array}\:,
	&1\leq j \leq s-1\\
0, &\text{otherwise}
\end{cases}
\end{equation*}
A similar computation gives the result for $\tilde{e}_0^j$.
\end{proof}

Now we are ready to prove the following:
\begin{theorem}\label{RMatrix}
On the $D_n$ highest weight vectors in $B^{2,s}\otimes B^{2,1}, s\geq 1,$ we have:
\begin{equation}
\mathcal{R}(u_k\otimes b)=\left \{ \begin{array}{ll}
u_1\otimes u_s, &
\text{ if } b= u_1 \text{ and }k=s,
\\
\varnothing \otimes u_s, &
\text{ if } b= u_1 \text{ and }k=s-1,\vspace{1mm}
\\
u_1\otimes \begin{array}{|c|c|}\hline1^{k+1} & \overline{2}\\ \hline 2^{k+1}& \overline{1}\\ \hline\end{array}\:, 
&\text{ if } b= u_1 \text{ and }0\leq k\leq s-2,\vspace{1mm}
\\
u_1\otimes \begin{array}{|c|c|}\hline1^{s-1} & 1\\ \hline 2^{s-1}& 3\\ \hline\end{array}\:, 
&\text{ if } b= \begin{array}{|c|}\hline 1\\ \hline 3\\ \hline \end{array} \text{ and }k=s,\vspace{1mm}
\\
u_1\otimes \begin{array}{|c|c|c|}\hline1^{k-1} & 1& 3\\ \hline 2^{k-1}& 3 & \overline{3}\\ \hline\end{array}\:, 
&\text{ if } b= \begin{array}{|c|}\hline 1\\ \hline 3\\ \hline \end{array} \text{ and }1\leq k\leq s-1,\vspace{1mm}
\\
u_1\otimes \begin{array}{|c|c|}\hline1^{k-1} & 3\\ \hline 2^{k-1}& 4\\ \hline\end{array}\:, 
&\text{ if } b= \begin{array}{|c|}\hline 3\\ \hline 4\\ \hline \end{array} \text{ and }1\leq k\leq s,\vspace{1mm}
\\
u_1\otimes \begin{array}{|c|c|}\hline1^{k-1} & 3\\ \hline 2^{k-1}& \overline{3}\\ \hline\end{array}\:, 
&\text{ if } b= \begin{array}{|c|}\hline 3\\ \hline \overline{3}\\ \hline \end{array} \text{ and }1\leq k\leq s,\vspace{1mm}
\\
u_1\otimes \begin{array}{|c|c|}\hline1^{k-1} & 1\\ \hline 2^{k-1}& \overline{2}\\ \hline\end{array}\:, 
&\text{ if } b= \begin{array}{|c|}\hline 1\\ \hline \overline{2}\\ \hline \end{array} \text{ and }1\leq k \leq s,\vspace{1mm}
\\
u_1\otimes \begin{array}{|c|c|}\hline1^{k-2} & 1\\ \hline 2^{k-2}& 3\\ \hline\end{array}\:, 
&\text{ if } b= \begin{array}{|c|}\hline 3\\ \hline \overline{2}\\ \hline \end{array} \text{ and }2\leq k\leq s,\vspace{1mm}
\\
u_1\otimes u_{s-1}, &
\text{ if } b= \varnothing \text{ and }k=s,\vspace{1mm}
\\
\varnothing \otimes u_{k}, &
\text{ if } b= \varnothing \text{ and }0\leq k\leq s-1,\vspace{1mm}
\\
u_1 \otimes \begin{array}{|c|}\hline \overline{2}\\ \hline \overline{1}\\ \hline \end{array}\:,&  \text{ if } b= \begin{array}{|c|}\hline \overline{2}\\ \hline \overline{1}\\ \hline \end{array} \text{ and }k=1,\vspace{1mm}
\\
u_1 \otimes u_{k-2},&  \text{ if } b= \begin{array}{|c|}\hline \overline{2}\\ \hline \overline{1}\\ \hline \end{array} \text{ and }2 \leq k\leq s,\vspace{1mm}
\\
u_1\otimes \begin{array}{|c|c|}\hline1^{k-1} & 3\\ \hline 2^{k-1}& \overline{4}\\ \hline\end{array}\:, 
&\text{ if } b= \begin{array}{|c|}\hline 3\\ \hline \overline{4}\\ \hline \end{array} \text{ and }1\leq k\leq s,n=4.
\end{array}\right .
\end{equation}
\end{theorem}
\begin{proof}
Since $\mathcal{R}$ is a $D_n^{(1)}$ crystal isomorphism, we have $\text{wt}(b\otimes b')=\text{wt}(\mathcal{R}(b\otimes b'))$.  Since $u_s\otimes u_1$ is the unique highest weight vector of $B^{2,s}\otimes B^{2,1}$ such that $\text{wt}(b\otimes b')=(s+1)\Lambda_2$, we must have $\mathcal{R}(u_s\otimes u_1)=u_1\otimes u_s.$  Therefore $\mathcal{R}(\tilde{e}_0^j(u_s \otimes u_1))=\tilde{e}_0^j\mathcal{R}(u_s\otimes u_1)=\tilde{e}_0^j(u_1\otimes u_s),j\geq 0.$  By the previous Lemma, and Definition \ref{tensor}:
\begin{equation*}
\tilde{e}_0^j(u_s \otimes u_1)=\left \{ \begin{array}{ll}
			u_s\otimes \varnothing,& \text{if } j=1\\ 
			u_{s-j+2} \otimes \begin{array}{|c|}\hline \overline{2} \\ \hline \overline{1}\\ \hline \end{array}\:,& \text{if }2\leq j \leq s+2\vspace{1mm} \\
\begin{array}{|c|}\hline \overline{2}^{j-s-2} \\ \hline \overline{1}^{j-s-2}\\ \hline \end{array} \otimes \begin{array}{|c|}\hline \overline{2} \\ \hline \overline{1}\\ \hline \end{array}\:,& \text{if }s+3\leq j \leq 2s+2
\end{array}\right .
\end{equation*}
and,
\begin{equation*}
\tilde{e}_0^j(u_1 \otimes u_s)=\left \{
\begin{array}{ll}
u_1\otimes u_{s-j},& \text{if } 1\leq j\leq s\vspace{1mm} \\
u_1 \otimes \begin{array}{|c|}\hline \overline{2}^{j-s} \\ \hline\overline{1}^{j-s}\\ \hline \end{array}\:, &\text{if }s+1\leq j \leq 2s 
\vspace{1mm} \\
\varnothing \otimes \begin{array}{|c|}\hline \overline{2}^{s} \\ \hline\overline{1}^{s}\\ \hline \end{array}\:, &\text{if }j=2s+1
\vspace{1mm} \\
\begin{array}{|c|}\hline \overline{2} \\ \hline\overline{1} \\ \hline \end{array} \otimes \begin{array}{|c|}\hline \overline{2}^{s} \\ \hline\overline{1}^{s}\\ \hline \end{array}\:, &\text{if }j=2s+2
\end{array}\right .
\end{equation*}
Setting $k=s-j+1$, we see:
\begin{eqnarray*}
\mathcal{R}(u_s\otimes \varnothing)=u_1\otimes u_{s-1},
\end{eqnarray*}
and, setting $k=s-j+2$ gives:
\begin{equation*}
\mathcal{R}\left (u_{k} \otimes \begin{array}{|c|}\hline \overline{2} \\ \hline \overline{1}\\ \hline \end{array}\:\right )=\left \{\begin{array}{ll}
	u_1\otimes u_{k-2},& \text{if }2\leq k \leq s\\ \\
	u_1\otimes \begin{array}{|c|}\hline \overline{2} \\ \hline \overline{1}\\ \hline \end{array}\;, &\text{if }k=1
\end{array}\right .
\end{equation*}
Also, we have 
\begin{equation*}\mathcal{R}\left (\begin{array}{|c|}\hline \overline{2}^{k} \\ \hline \overline{1}^{k}\\ \hline \end{array} \otimes \begin{array}{|c|}\hline \overline{2} \\ \hline \overline{1}\\ \hline \end{array}\: \right )=
\left \{\begin{array}{ll}u_1 \otimes \begin{array}{|c|}\hline \overline{2}^k \\ \hline \overline{1}^k\\ \hline \end{array}\:,
&0\leq k \leq s-2\vspace{1mm} \\
\varnothing \otimes \begin{array}{|c|}\hline \overline{2}^s \\ \hline \overline{1}^{s}\\ \hline \end{array}\:,& k =s-1\vspace{1mm} \\
\begin{array}{|c|}\hline\overline{2} \\ \hline \overline{1} \\\hline \end{array}\otimes \begin{array}{|c|}\hline \overline{2}^s \\ \hline \overline{1}^s\\ \hline \end{array}\:,
& k =s
\end{array}
\right .
\end{equation*}
however, these are not $D_n$ highest weight vectors.  By using column insertion (Theorem \ref{insertion}) we find the corresponding $D_n$ highest weight vectors, and obtain: $\mathcal{R}\left (u_k \otimes u_1\: \right )=u_1 \otimes \begin{array}{|c|c|}\hline 1^{k+1}&\overline{2} \\ \hline 2^{k+1}&\overline{1}\\ \hline \end{array}\:,0\leq k \leq s-2$ (and nothing else new).

Now, we consider the case $\varnothing \otimes T$, where $T=\begin{array}{|c|}\hline T_{1,1} \\ \hline T_{2,1} \\ \hline \end{array}\:=\begin{array}{|c|}\hline 1 \\ \hline 3 \\ \hline \end{array}\:,\begin{array}{|c|}\hline 3 \\ \hline 4 \\ \hline \end{array}\:,\begin{array}{|c|}\hline 3 \\ \hline \overline{3} \\ \hline \end{array}\:,\begin{array}{|c|}\hline 1 \\ \hline \overline{2} \\ \hline \end{array}
\:,\begin{array}{|c|}\hline 3 \\ \hline \overline{2} \\ \hline \end{array}\:,$ or $\begin{array}{|c|} \hline 3\\ \hline \overline{4} \\ \hline \end{array}\:,$ if $n=4$.  We see that the corresponding highest weight vector is $\varnothing \otimes u_1$.  We have seen that $\mathcal{R}(\varnothing \otimes u_1)=u_1 \otimes \begin{array}{|c|c|}\hline 1&\overline{2} \\ \hline 2&\overline{1}\\ \hline \end{array}\:$, which gives (using inverse column insertion on the tableau $T$): $\mathcal{R}(\varnothing\otimes T)=u_1\otimes \begin{array}{|c|c|}\hline T_{1,1}&\overline{2} \\ \hline T_{2,1}&\overline{1}\\ \hline \end{array}\:$.  Acting on both sides by $\tilde{f}_0$ a sufficient number of times using the previous Lemma, and commuting with $\mathcal{R}$ gives:
\begin{equation*}
\mathcal{R} (u_k\otimes T)=\left \{\begin{array}{ll}
u_1\otimes \begin{array}{|c|c|}\hline1^{s-1} & 1\\ \hline 2^{s-1}& 3\\ \hline\end{array}\:, 
&\text{ if } T= \begin{array}{|c|}\hline 1\\ \hline 3\\ \hline \end{array} \text{ and }k=s,\vspace{1mm}
\\
u_1\otimes \begin{array}{|c|c|c|}\hline1^{k-1} & 1& 3\\ \hline 2^{k-1}& 3 & \overline{3}\\ \hline\end{array}\:, 
&\text{ if } T= \begin{array}{|c|}\hline 1\\ \hline 3\\ \hline \end{array} \text{ and }1\leq k\leq s-1,\vspace{1mm}
\\
u_1\otimes \begin{array}{|c|c|}\hline1^{k-1} & 3\\ \hline 2^{k-1}& 4\\ \hline\end{array}\:, 
&\text{ if } T= \begin{array}{|c|}\hline 3\\ \hline 4\\ \hline \end{array} \text{ and }1\leq k\leq s,\vspace{1mm}
\\
u_1\otimes \begin{array}{|c|c|}\hline1^{k-1} & 3\\ \hline 2^{k-1}& \overline{3}\\ \hline\end{array}\:, 
&\text{ if } T= \begin{array}{|c|}\hline 3\\ \hline \overline{3}\\ \hline \end{array} \text{ and }1\leq k\leq s,\vspace{1mm}
\\
u_1\otimes \begin{array}{|c|c|}\hline1^{k-1} & 1\\ \hline 2^{k-1}& \overline{2}\\ \hline\end{array}\:, 
&\text{ if } T= \begin{array}{|c|}\hline 1\\ \hline \overline{2}\\ \hline \end{array} \text{ and }1\leq k \leq s,\vspace{1mm}
\\
u_1\otimes \begin{array}{|c|c|}\hline1^{k-2} & 1\\ \hline 2^{k-2}& 3\\ \hline\end{array}\:, 
&\text{ if } T= \begin{array}{|c|}\hline 3\\ \hline \overline{2}\\ \hline \end{array} \text{ and }2\leq k\leq s,\vspace{1mm} \\
u_1\otimes \begin{array}{|c|c|}\hline1^{k-1} & 3\\ \hline 2^{k-1}& \overline{4}\\ \hline\end{array}\:, 
&\text{ if } T= \begin{array}{|c|}\hline 3\\ \hline \overline{4}\\ \hline \end{array} \text{ and }1\leq k\leq s,n=4.
\end{array}
\right .
\end{equation*}
Finally, consider $\begin{array}{|c|}\hline \overline{2}^s\\ \hline \overline{1}^s \\ \hline \end{array} \otimes \varnothing$.  The corresponding highest weight vector is $u_s \otimes \varnothing$.  We have seen that $\mathcal{R}(u_s \otimes \varnothing)=u_1\otimes u_{s-1},$ which gives $\mathcal{R}\left (\begin{array}{|c|}\hline \overline{2}^s\\ \hline \overline{1}^s \\ \hline \end{array} \otimes \varnothing\right )=\begin{array}{|c|}\hline \overline{2}\\ \hline \overline{1} \\ \hline \end{array}\otimes \begin{array}{|c|}\hline \overline{2}^{s-1}\\ \hline \overline{1}^{s-1} \\ \hline \end{array}\:$. Acting on both sides by $\tilde{f}_0$ sufficiently many times, and commuting with $\mathcal{R}$ gives:
\begin{equation*}
\mathcal{R}(u_k\otimes T)=\left \{\begin{array}{ll}
\varnothing \otimes u_k, &\text{if }T=\varnothing, 0\leq k\leq s-1\\
\varnothing \otimes u_s, &\text{if }T=u_1, k=s-1.
\end{array}\right .
\end{equation*}
\end{proof}
From the computation in the proof of Theorem \ref{RMatrix} one may easily deduce the following:
\begin{corollary}\label{energy}
For the highest weight vector $u_k\otimes T\in B^{2,s}\otimes B^{2,1},0\leq k\leq s$,
\begin{equation}
H(u_k\otimes T)=\left \{ \begin{array}{ll}
0, & \text{if }T=\begin{array}{|c|}\hline 1 \\ \hline 2\\ \hline \end{array}\:, k=s\vspace{1mm}\\
-1,& \text{if }T=\begin{array}{|c|}\hline 1 \\ \hline 2\\ \hline \end{array}\:, k = s-1,\vspace{1mm}
		\\
		& T=\begin{array}{|c|}\hline 1 \\ \hline 3\\ \hline \end{array}\:, k=s, \text{ or},\\
		& T=\varnothing, k=s,\\
		-2, & \text{otherwise.}
\end{array}
\right .
\end{equation}
\end{corollary}
Similarly, we have:
\begin{theorem}\label{RHighest2}
On the $D_n$ highest weight vectors in $B^{1,1}\otimes B^{2,1}$,
\begin{equation}
\mathcal{R}\left (\:\begin{array}{|c|}\hline 1 \\ \hline \end{array}\otimes T\right)=\begin{cases}
			\begin{array}{|c|}\hline 1 \\ \hline 2\\ \hline \end{array}\: \otimes \:
			\begin{array}{|c|}\hline 1 \\ \hline \end{array}\:, &\text{if }T=\begin{array}{|c|}\hline 1 \\ \hline 2\\ \hline
			\end{array}\:,\vspace{1mm}
			\\
			\begin{array}{|c|}\hline 1 \\ \hline 2\\ \hline \end{array}\: \otimes \:
			\begin{array}{|c|}\hline 3 \\ \hline \end{array}\:, &\text{if }T=\begin{array}{|c|}\hline 2 \\ \hline 3\\ \hline
			\end{array}\:,\vspace{1mm}
			\\
			\begin{array}{|c|}\hline 1 \\ \hline 2\\ \hline \end{array}\: \otimes \:
			\begin{array}{|c|}\hline \overline{2} \\ \hline \end{array}\:, &\text{if }T=\varnothing,\vspace{1mm}
			\\
			\varnothing\: \otimes \:
			\begin{array}{|c|}\hline 1 \\ \hline \end{array}\:, &\text{if }T=\begin{array}{|c|}\hline 2 \\ \hline \overline{2}\\ 
			\hline
			\end{array}\:.
		\end{cases}
\end{equation}
\end{theorem}
\begin{proof}
The cases where $T=\begin{array}{|c|}\hline 1 \\ \hline 2\\ \hline\end{array}$  or $\begin{array}{|c|}\hline 2 \\ \hline 3\\ \hline\end{array}$ are immediate because there is no outer multiplicity in these cases.  Consider $\mathcal{R}\left (\:\begin{array}{|c|}\hline 1 \\ \hline \end{array}\otimes \varnothing \right).$  We have:
\begin{eqnarray*}
\tilde{f}_0\left (\:\begin{array}{|c|}\hline 1 \\ \hline \end{array}\:\otimes \varnothing \right)
&=&\begin{array}{|c|}\hline 1 \\ \hline \end{array}\:\otimes \:\begin{array}{|c|}\hline 1 \\ \hline 2\\ \hline \end{array}\\
\tilde{f}_0\mathcal{R}\left (\:\begin{array}{|c|}\hline 1 \\ \hline \end{array}\:\otimes \varnothing \right)
&=&\begin{array}{|c|}\hline 1 \\ \hline 2\\ \hline \end{array} \:\otimes \: \begin{array}{|c|}\hline 1 \\ \hline \end{array}\\
\mathcal{R}\left (\:\begin{array}{|c|}\hline 1 \\ \hline \end{array}\:\otimes \varnothing \right)
&=&\tilde{e}_0\left (\:\begin{array}{|c|}\hline 1 \\ \hline 2\\ \hline \end{array} \:\otimes \: \begin{array}{|c|}\hline 1 \\ \hline \end{array}\:\right )\\
&=&\begin{array}{|c|}\hline 1 \\ \hline 2\\ \hline \end{array} \:\otimes \: \begin{array}{|c|}\hline \overline{2} \\ \hline \end{array}\\
\end{eqnarray*}
Finally, the case where $T=\begin{array}{|c|}\hline 2 \\ \hline \overline{2}\\ \hline \end{array}$ is immediate, since we have ruled out all other possibilities.
\end{proof}
\subsection{Combinatorial $R$ matrix for $B^{2,s}\otimes B^{2,1}$}
Recall (Theorem \ref{insertion}) that there is a bijection between the set $B(\Lambda_1)^{\otimes l}$ and the set of all pairs $(P,Q)$  where $Q=(Q_0=(\varnothing,0),Q_1,Q_2,\dots, Q_l)$ is an oscillating tableau, and $P$ is a Kashiwara-Nakashima tableau of shape $O_l$.  In fact (\cite{L}), the oscillating tableau $Q$ enables us to determine the highest-weight vector in $B(\Lambda_1)^{\otimes l}$ corresponding to the pair $(P,Q)$: namely it is the vector $b_1\otimes b_2 \otimes \cdots \otimes b_l$, where
\begin{equation*}
b_k=\begin{cases} i, &\text{if }1\leq i < n, O_k \text{ has one more box in the $i$th row than }O_{k-1}\\
			\overline{i}, & \text{if }1\leq i < n, O_k\text{ has one fewer box in the $i$th row than }O_{k-1}\\
			n, & \text{if }\varepsilon_k=+,0 \text{(resp. $-$) and $O_k$ has 1 more (resp. fewer) box }\\
			&\text{in the $n$th row than }O_{k-1}\\
			\overline{n}, &\text{otherwise}.
\end{cases}
\end{equation*}
The following procedure then allows us to compute the combinatorial $R$-matrix for any element $T\otimes T’\in B^{2,s}\otimes B^{2,1}$.
\begin{enumerate}
\item Insert the letters of $T’$ from top to bottom into $T$, keeping track of the rows in which boxes are being added or removed and the positions of any $n$ or $\overline{n}$ in a column of height $n$.  If there $T’=\varnothing$ then do nothing.  Call the resulting tableau $P$.
\item Find the highest weight vector $b_1\otimes b_2\otimes \cdots \otimes b_l$ from Step 1.  The result gives the reading of a unique highest weight vector $u_k\otimes T''\in B^{2,s}\otimes B^{2,1}.$
\item Compute $\mathcal{R}(u_k\otimes T''),$ and interpret the result as an element of $B(\Lambda_1)^{\otimes l’}$, and use it to deterine the oscillating tableau $Q'.$
\item For the pair $(P,Q')$, reverse Lecouvey's algorithm to get a sequence of letters.  The resulting sequence of letters gives the reading of a unique pair of tableaux $\widetilde{T'}\otimes \widetilde{T}\in B^{2,1}\otimes B^{2,s}$ which is $\mathcal{R}(T\otimes T').$
\end{enumerate}

\emph{Example:} We compute $\mathcal{R}\left (\:\begin{array}{|c|}\hline \overline{4}\\ \hline 4\\ \hline \end{array}\otimes \begin{array}{|c|}\hline 1\\ \hline 2\\ \hline \end{array}\:\right )$ for the $D_4$-crystal $B^{2,2}\otimes B^{2,1}.$  We insert as follows: $T'=2\to \left (1 \to \begin{array}{|c|}\hline \overline{4}\\ \hline 4\\ \hline \end{array}\:\right )=\begin{array}{|c|c|}\hline 1& \overline{4}\\ \hline 2& 4\\ \hline \end{array}\:$.  The corresponding highest weight vector is $1\otimes 2 \otimes 1 \otimes 2,$ which corresponds to $u_1\otimes u_1\in B^{2,2}\otimes B^{2,1}.$  By Theorem \ref{RMatrix} we have $\mathcal{R}(u_1\otimes u_1)=\varnothing \otimes u_2=\varnothing \otimes \begin{array}{|c|c|}\hline 1 & 1\\ \hline 2 & 2\\ \hline \end{array}\:.$  So, we remove letters from the tableau $T'$ successively from rows $2,1,2,1$ which gives $\overline{4}\otimes 4\otimes 1\otimes 2,$ which, in the $B(0)\otimes B(2\Lambda_2)$ component of $B^{2,1}\otimes B^{2,2}$ is interpreted as $\varnothing \otimes \begin{array}{|c|c|}\hline 1 & \overline{4}\\ \hline 2 & 4\\ \hline \end{array}\:.$

\section{Soliton Cellular Automaton}
We define $\mathcal{P}_L=\{b_1\otimes b_2 \otimes \cdots \otimes b_L\in(B^{2,1})^{\otimes L}| T_n=u_1, \text{ for } n \text{  sufficiently large}\}$ to be the set of \emph{states} of the $D_n^{(1)}$ soliton cellular automaton.  We depict the operation $\mathcal{R}(b\otimes b')=\tilde{b}'\otimes \tilde{b}$ by:
\setlength{\unitlength}{.5in}
\begin{center}
\begin{picture}(1.4,1.4)(-.5,-1)
\put(-.05,.05){$b$}
\put(0,0){\vector(0,-1){1}}
\put(-.05,-1.35){$\tilde{b}$}
\put(-.65,-.55){$b'$}
\put(-.5,-.5){\vector(1,0){1}}
\put(.55,-.55){$\tilde{b}'$}
\end{picture}
\end{center}
\vspace{.2in}
 Fix $s>0$ and let $u_s\in B^{2,s}$ be the highest weight vector for $D_n$.  For $p=b_1\otimes b_2 \otimes \cdots \otimes b_L\in \mathcal{P}_L$ we define the \emph{time evolution operator} $T_l(p)$:
\begin{equation}
T_l(p)\otimes u_l=\mathcal{R}_{L \: L+1}\mathcal{R}_{L-1 \: L}\cdots\mathcal{R}_{23} \mathcal{R}_{12}(u_l \otimes p). \label{ts}
\end{equation}
The transition of phase $T_l(b_1\otimes b_2 \otimes \cdots \otimes b_L)=\tilde{b}_1\otimes \tilde{b}_2 \otimes \cdots \otimes \tilde{b}_L$ is depicted as
\begin{center}
\begin{picture}(6,1.4)(-.5,-1)
\put(-.05,.05){$b_1$}
\put(0,0){\vector(0,-1){1}}
\put(-.05,-1.35){$\tilde{b}_1$}
\put(-1.7,-.55){$u^{(0)}=u_l$}
\put(-.5,-.5){\vector(1,0){1}}
\put(.55,-.55){$u^{(1)}$}
\put(1.45,.05){$b_2$}
\put(1.5,0){\vector(0,-1){1}}
\put(1.45,-1.35){$\tilde{b}_2$}
\put(1,-.5){\vector(1,0){1}}
\put(2.05,-.55){$u^{(2)}\cdots u^{(L-1)}$}
\put(4.45,.05){$b_L$}
\put(4.5,0){\vector(0,-1){1}}
\put(4.45,-1.35){$\tilde{b}_L$} 
\put(4,-.5){\vector(1,0){1}}
\put(5.05,-.55){$u^{(L)}=u_l.$}
\end{picture}
\end{center}
\vspace{.2in}
We define the \emph{state energy} to be the sum:
\begin{equation}
E_l(p)=-\sum_{i=0}^{L-1} H(u^{(i)}\otimes b_{i+1}).
\end{equation}
We can also use the combinatorial $R$-matrix $\mathcal{R}:B^{1,1} \otimes B^{2,1} \to B^{2,1} \otimes B^{1,1}$  to define an operator $T_\natural$ similar to $T_l$, by
	\begin{equation}
		T_\natural(p)\otimes b(p)=\mathcal{R}_{L \: L+1}\mathcal{R}_{L-1 \:L}
		\cdots\mathcal{R}_{23}\mathcal{R}_{12}\left (\:\begin{array}{|c|}\hline 1 \\ \hline \end{array} \otimes p\right )\:. \label{tn}
	\end{equation}
In this case, $b(p)$ is dependent on the state $p$ so we indicate this dependence in the definition.

\subsection{$D_n^{(1)}$-solitons and their scattering rules}
Experience from other soliton cellular automata has shown that states $p\in\mathcal{P}_L$ satisfying $E_1(p)=1$ correspond to the so-called ``one-soliton states''. 
\begin{proposition}\label{onesol}
In the $D_n^{(1)}$ SCA, $E_1(p)=1$ if and only if $p\neq u_1^{\otimes L}$ and is of the following form:
\begin{equation}\label{1soliton}
u_1^{\otimes i}\otimes 
\:\begin{array}{|c|}\hline 2\\ \hline b_1\\ \hline\end{array}\:\otimes
\: \begin{array}{|c|}\hline 2\\ \hline b_2\\ \hline\end{array}\: \otimes
\cdots
\otimes
\:\begin{array}{|c|}\hline 2\\ \hline b_j\\ \hline\end{array}\: \otimes
\:\begin{array}{|c|}\hline 1\\ \hline b_{j+1}\\ \hline\end{array}\: \otimes
\:\begin{array}{|c|}\hline 1\\ \hline b_{j+2}\\ \hline\end{array}\: \otimes
\cdots
\otimes
\:\begin{array}{|c|}\hline 1\\ \hline b_k\\ \hline\end{array}\: \otimes
u_1^{\otimes l}
\end{equation}
for some $i,j\leq k,l \in \mathbb{Z}_{\geq 0}$ such that $i+k+l=L$, and some $b_1\geq b_{2}\geq \cdots \geq b_{k}\in B\backslash \{1,2,\overline{2},\overline{1}\}.$  
\end{proposition}
\begin{proof}
Let $p=T_1\otimes T_2\otimes \cdots \otimes T_L\in \mathcal{P}_L$ be such that $E_1(p)=1$.  For $T\otimes T' \in B^{2,1} \otimes B^{2,1}$ it is the case that $\mathcal{R}(T \otimes T')=T\otimes T'.$  Thus, $E_1(p)=-H(u_1\otimes T_1)-\sum_{i=1}^L H(T_i\otimes T_{i+1})=1$, and all the terms appearing in this sum are $\geq 0$ by Corollary \ref{energy}.  Now, $u_1\otimes T=0$ if and only if $T=\:\begin{array}{|c|}\hline 1\\ \hline 2\\ \hline\end{array}\:$.  Suppose that we have $H(T_k\otimes T_{k+1})=-1$ for some $T_{k}=u_1$.  In our SCA, $T_k=u_1, k\gg0$ and $H(\varnothing \otimes T')<0, T'\in B^{2,1},$ hence it must not be the case that $T_{k+1}=\varnothing.$  So we are left in the case that $u_1 \otimes T$ has $D_n$ highest weight vector $u_1\otimes \:\begin{array}{|c|}\hline 1\\ \hline 3\\ \hline\end{array}\:$.  By column insertion, (Theorem \ref{insertion}) this is only the case if $T=\:\begin{array}{|c|}\hline 1\\ \hline b\\ \hline\end{array}\:$, or $\begin{array}{|c|}\hline 2\\ \hline b\\ \hline\end{array}\:$, where $b\in B\backslash \{1,2,\overline{2},\overline{1}\}$.

Since $E_1(p)=1$, we must have $H(T_j\otimes T_{j+1})=0, k<j<L.$  Supposing $T_j$ to be of the form $\begin{array}{|c|}\hline b_1\\ \hline b_2\\ \hline\end{array}\:,b_1\in \{1,2\}$, we see that $T_{j+1}$ must be $\:\begin{array}{|c|}\hline b_1'\\ \hline b_2'\\ \hline\end{array}\:, b_1' \leq b_1, b_2'\leq b_2,b_2'\in B\backslash \{1,\overline{2},\overline{1}\}$ in order for $T_{j}\otimes T_{j+1}$ to have $D_n$ highest weight vector $u_1 \otimes u_1$.  Finally, $T_j=u_1=\begin{array}{|c|}\hline 1\\ \hline 2\\ \hline\end{array}\:, j\gg 0$ for our SCA.  
\end{proof}
The proof of the following proposition is completely analogous to that of Proposition 9 in \cite{MOW}, except that $\varepsilon_i(u_l),$ and $\varphi_i(u_l)$ are now 0. 
\begin{proposition}\label{commute}
Let $p\in \mathcal{P}_L$.  If $\tilde{e}_i(p)\neq 0$ then $\tilde{e}_iT_l(p)=T_l(\tilde{e}_i(p)),i\neq 0,2$ and $E_l(\tilde{e}_i(p))=E_l(p)$, otherwise $\tilde{e}_iT_l(p)=0,i\neq 0,2$.  The same relations hold for $\tilde{f}_i,i\neq 0,2$.  
\end{proposition}
Recall the $A_1^{(1)}$-crystal $\widehat{B}^{1,s}=\{(x_1,x_2)\in \mathbb{Z}_{\geq 0}^2| x_1+x_2=s\}$, where $(x_1,x_2)$ can be associated with the set of tableaux: 
$$\begin{array}{l @{} l}\underbrace{\begin{array}{|c|c|c|c|}\hline 1' & 1' & \cdots & 1' \\ \hline\end{array}}_{x_1}& 
\underbrace{\begin{array}{c|c|c|c|}\hline  2' & 2' & \cdots & 2'\\ \hline \end{array}}_{x_2}\end{array},$$
and, similarly, the $D_n^{(1)}$-crystal ${B}^{1,s}=\{(x_1,x_2,\cdots, \overline{x}_1)\in \mathbb{Z}_{\geq 0}^{2n}| x_1+x_2+\cdots+\overline{x}_1=s,x_n=0 \text{ or }\overline{x}_n=0\}$ can be associated with the tableaux: 
$$\begin{array}{l @{} l @{} l @{} l @{} l @{} l}\underbrace{\begin{array}{|c|c|c|c|}\hline 1 & 1 & \cdots & 1 \\ \hline\end{array}}_{x_1}& 
\underbrace{\begin{array}{c|c|c|c|}\hline  2 & 2 & \cdots & 2\\ \hline \end{array}}_{x_2}&\begin{array}{|c|}\hline \cdots\\ \hline \end{array}&\underbrace{\begin{array}{c|c|c|c|}\hline  \overline{1} & \overline{1} & \cdots & \overline{1}\\ \hline \end{array}}_{\overline{x}_1}\end{array}.$$

\begin{proposition} \label{bijection}
Let $(T,T')$ be an element of the $A_1^{(1)}\oplus D_{n-2}^{(1)}$-crystal $\widehat{B}^{1,s}\times {B}^{1,s},n>5$.  Define the map $i_s:\widehat{B}^{1,s}\times {B}^{1,s}\to (B^{2,1})^{\otimes s}$ given by 
$$\big (\:\begin{array}{|c|c|}\hline 1'^{j}&2'^{s-j}\\ \hline\end{array}\:, \:\begin{array}{|c|c|c|c|}\hline b_1& b_{2}&\cdots& b_{s} \\ \hline\end{array}\:\big )\mapsto
$$
$$
\begin{array}{|c|}\hline 2\\ \hline b_s'\\ \hline\end{array}\:\otimes
\: \begin{array}{|c|}\hline 2\\ \hline b_{s-1}'\\ \hline\end{array}\: \otimes
\cdots
\otimes
\:\begin{array}{|c|}\hline 2\\ \hline b_{s-j}'\\ \hline\end{array}\: \otimes
\:\begin{array}{|c|}\hline 1\\ \hline b_{s-j-1}'\\ \hline\end{array}\: \otimes
\:\begin{array}{|c|}\hline 1\\ \hline b_{s-j-2}'\\ \hline\end{array}\: \otimes
\cdots
\otimes
\:\begin{array}{|c|}\hline 1\\ \hline b_1'\\ \hline\end{array}$$
where,
$$
b_k'=\begin{cases}b_k+2,&\text{ if }1\leq b_k \leq n-2\\
		\overline{c+2},& \text{ if }b_k=\overline{c},1\leq c \leq n-2\end{cases}
$$
Then $i_s$ satisfies the relations:
$$i_s(\tilde{e}_1^AT,T')= \tilde{e}_1 i_s(T,T'),i_s(\tilde{f}_1^AT,T')=\tilde{f}_1i_s(T,T'),$$
$$i_s(T,\tilde{e}_i^DT')= \tilde{e}_{i+2} i_s(T,T'), i_s(T,\tilde{f}_i^DT')=\tilde{f}_{i+2}i_s(T,T'), 1\leq i\leq n-2,$$
If $n=5$ then we have:
$$
\begin{array}{|c|}\hline 1'\\ \hline 2'\\ \hline\end{array}\mapsto \begin{array}{|c|}\hline 3\\ \hline \end{array}\;, 
\begin{array}{|c|}\hline 1'\\ \hline 3'\\ \hline\end{array}\mapsto \begin{array}{|c|}\hline 4\\ \hline \end{array}\:,
\begin{array}{|c|}\hline 2'\\ \hline 3'\\ \hline\end{array}\mapsto \begin{array}{|c|}\hline 5\\ \hline \end{array}\:,
\begin{array}{|c|}\hline 1'\\ \hline 4'\\ \hline\end{array}\mapsto \begin{array}{|c|}\hline \overline{5}\\ \hline \end{array}\:,
\begin{array}{|c|}\hline 2'\\ \hline 4'\\ \hline\end{array}\mapsto \begin{array}{|c|}\hline \overline{4}\\ \hline \end{array}\:,
\begin{array}{|c|}\hline 3'\\ \hline 4'\\ \hline\end{array}\mapsto \begin{array}{|c|}\hline \overline{3}\\ \hline \end{array}\:, 
$$
yields a map $i_s$ satisfying 
$$i_s(\tilde{e}_1^AT,T')= \tilde{e}_1 i_s(T,T'),i_s(\tilde{f}_1^AT,T')=\tilde{f}_1i_s(T,T'),$$
$$i_s(T,\tilde{e}_i^AT')= \tilde{e}_{\sigma(i)+2} i_s(T,T'), i_s(T,\tilde{f}_i^AT')=\tilde{f}_{\sigma(i)+2}i_s(T,T'), 1\leq i \leq 3$$
where $\sigma$ transposes 1 and 2 and fixes 3.

For $n=4$ we have:
$$
((x_1,x_2),(y_1,y_2))\mapsto \overline{3}^{s-\max(x_1,y_1)}\otimes \overline{4}^{(x_1-y_1)_+}\otimes 4^{(y_1-x_1)_+}\otimes 3^{\min(x_1,y_1)} 
$$
where $(x)_+=\max(x,0),$ yields a map $i_s$ satisfying:
$$i_s(\tilde{e}_1^AT,T',T'')= \tilde{e}_1 i_s(T,T',T''),i_s(\tilde{f}_1^AT,T',T'')=\tilde{f}_1i_s(T,T',T''),$$
$$i_s(T,\tilde{e}_1^AT',T'')= \tilde{e}_3 i_s(T,T',T''),i_s(T,\tilde{f}_1^AT',T'')=\tilde{f}_3i_s(T,T',T''),$$
$$i_s(T,T',\tilde{e}_1^AT'')= \tilde{e}_4 i_s(T,T',T''),i_s(T,T',\tilde{f}_1^AT'')=\tilde{f}_4i_s(T,T',T'').$$
\end{proposition}
\begin{proof}
We have:
\begin{equation*}
i_s(\tilde{e}_1^A\begin{array}{|c|c|}\hline 1'^{j}&2'^{s-j}\\ \hline\end{array}\:,T')= i_s(\tilde{e}_1^A(2^{\otimes s-j}\otimes 1^{\otimes j}),T'),
\end{equation*}
and 
\begin{multline*}
\tilde{e}_1 i_s\big (\:\begin{array}{|c|c|}\hline 1'^{j}&2'^{s-j}\\ \hline\end{array}\:,T' \big)=\\
\tilde{e}_1 \bigg (\:\begin{array}{|c|}\hline 2\\ \hline b_s'\\ \hline\end{array}\:\otimes
\: \begin{array}{|c|}\hline 2\\ \hline b_{s-1}'\\ \hline\end{array}\: \otimes
\cdots
\otimes
\:\begin{array}{|c|}\hline 2\\ \hline b_{s-j}'\\ \hline\end{array}\: \otimes
\:\begin{array}{|c|}\hline 1\\ \hline b_{s-j-1}'\\ \hline\end{array}\: \otimes
\:\begin{array}{|c|}\hline 1\\ \hline b_{s-j-2}'\\ \hline\end{array}\: \otimes
\cdots
\otimes
\:\begin{array}{|c|}\hline 1\\ \hline b_1'\\ \hline\end{array} \: \bigg ).
\end{multline*}
Observe $\tilde{e}_1$ only has non-zero action on the top row of $i_s\big (\:\begin{array}{|c|c|}\hline 1'^{j}&2'^{s-j}\\ \hline\end{array}\:,T' \big),$ since $b_k'\in \{3,4,\dots, n,\overline{n},\dots, \overline{4},\overline{3}\}.$  Also, the action of $\tilde{e}_1$ on the top row of $i_s\big (\:\begin{array}{|c|c|}\hline 1'^{j}&2'^{s-j}\\ \hline\end{array}\:\big )$ is equivalent to $\tilde{e}_1^A(2'^{\otimes s-j}\otimes 1'^{\otimes j})$, from which we may deduce that $i_s(\tilde{e}_1^AT,T')=\tilde{e}_1i_s(T,T')$.
 The proof is similar for $i_s(\tilde{f}_1^AT,T')=\tilde{f}_1i_s(T,T').$

The proof of $i_s(T,\tilde{e}_i^DT')= \tilde{e}_{i+2} i_s(T,T')$, and $i_s(T,\tilde{f}_i^DT')=\tilde{f}_{i+2}i_s(T,T'), 1\leq i\leq n-2$ is similar, except that $,\tilde{e}_{i+2}$ and $\tilde{f}_{i+2}$ act on the bottom row of $i_s(T,T')$. One easily checks that the actions of $\tilde{e}_{i+2}$ and $\tilde{f}_{i+2}$ on $b'_k$ are equivalent to $\tilde{e}_i^D$ and $\tilde{f}_i^D$ on $b_k$.
\end{proof}
\emph{Remark:} The above map is a bijection from $\widehat{B}^{1,s}\times {B}^{1,s},$ (or $\widehat{B}^{1,s}\times \widehat{B}^{2,s}$, or ($\widehat{B}^{1,s})^3)$ to $\{p\in\mathcal{P}_L|E_1(p)=1\}$ by Proposition \ref{onesol}.

A state of the following form is called and $m$-soliton state:
$$\dots [s_1]\dots\dots[s_2]\dots\cdots\dots[s_m]\dots$$
 where $s_1>s_2> \cdots > s_m$, $\dots[s]\dots$ denotes a local configuration of the form $\begin{array}{|c|}\hline 2\\ \hline b_1\\ \hline\end{array}\otimes
\: \begin{array}{|c|}\hline 2\\ \hline b_2\\ \hline\end{array}\: \otimes
\cdots
\otimes
\:\begin{array}{|c|}\hline 2\\ \hline b_j\\ \hline\end{array}\: \otimes
\:\begin{array}{|c|}\hline 1\\ \hline b_{j+1}\\ \hline\end{array}\: \otimes
\:\begin{array}{|c|}\hline 1\\ \hline b_{j+2}\\ \hline\end{array}\: \otimes
\cdots
\otimes
\:\begin{array}{|c|}\hline 1\\ \hline b_s\\ \hline\end{array}$
for some $b_1\geq b_{2}\geq \cdots \geq b_{s}\in B\backslash \{1,2,\overline{2},\overline{1}\}$, and the $[s_i]$ are separated by sufficiently many $u_1$'s.

\begin{proposition}
Let $p$ be a one-soliton state of length $s$.  Then 
\begin{enumerate}
\item $E_k(p)=\min(k,s)$,
\item $T_k(p)$ is obtained by rightward shift by $\min(k,s)$ lattice steps.
\end{enumerate}
\end{proposition}
\begin{proof}
By applying sufficient operators $\tilde{e}_i, i\neq 0,2$, $p$ becomes ${\begin{array}{|c|}\hline 1\\ \hline 3\\ \hline\end{array}\:}^{\otimes s}\otimes u_1\otimes u_1\otimes \cdots \otimes u_1$.  Using the combinatorial $R$-matrix we obtain:
\begin{eqnarray}
\mathcal{R}\bigg (\:\begin{array}{|c|c|}\hline  1^i & 1^{l-i} \\ \hline 2^i& 3^{s-i}\\ \hline \end{array}\:\otimes\: \begin{array}{|c|}\hline 1\\ \hline 3\\ \hline \end{array}\:\bigg )
	&=&u_1 \otimes \: \begin{array}{|c|c|}\hline  1^{i-1} & 1^{s-i+1} \\ \hline 2^{i-1} & 3^{s-i+1}\\ \hline\end{array} \text{ if } i>0,\label{A1}\\
\mathcal{R}\bigg (\:\begin{array}{|c|}\hline 1^s\\ \hline 3^{s} \\ \hline \end{array}\: \otimes \: \begin{array}{|c|}\hline 1\\ \hline 3\\ \hline \end{array}\: \bigg )
	&=&\begin{array}{|c|}\hline 1\\ \hline 3\\ \hline \end{array}\: \otimes \begin{array}{|c|}\hline 1^s \\ \hline  3^{s} \\ \hline \end{array}\label{A2}\\
\mathcal{R}\bigg (\:\begin{array}{|c|c|}\hline  1^i & 1^{s-i} \\ \hline 2^i & 3^{s-i}\\ \hline \end{array}\:\otimes u_1\: \bigg )
	&=&\begin{array}{|c|}\hline 1\\ \hline 3\\ \hline \end{array}\: \otimes \:  \begin{array}{|c|c|}\hline  1^{i+1} & 1^{s-i-1} \\ \hline 2^{i+1} & 3^{s-i-1} \\ \hline\end{array} \text{ if }i<s,\label{A3}\\
\mathcal{R} (u_s \otimes u_1)
	&=&u_1 \otimes u_s \label{A4}
\end{eqnarray}
and
\begin{eqnarray}
H\bigg (\:\begin{array}{|c|c|}\hline  1^i & 1^{s-i} \\ \hline 2^i & 3^{s-i} \\ \hline\end{array}\:\otimes \:\begin{array}{|c|}\hline 1\\ \hline 3\\ \hline \end{array}\:\bigg)
	&=&-1 \text{ if }i>0,\\
H\bigg (\begin{array}{|c|}\hline   1^{s} \\ \hline 3^s \\ \hline \end{array}\:\otimes \:\begin{array}{|c|}\hline 1\\ \hline 3\\ \hline \end{array}\:\bigg )
	&=&0 ,\\
H\bigg (\:\begin{array}{|c|c|}\hline  1^i & 1^{s-i} \\ \hline2^i & 3^{s-i}\\ \hline \end{array}\:\otimes u_1 \bigg)
	&=&0.
\end{eqnarray}
where, as before, the symbol $b^j$ means that $b$ is repeated $j$ times.  If $k<s$ then 
\begin{eqnarray*}
T_k(p)&=&\mathcal{R}_{L\; L+1}\cdots \mathcal{R}_{23}\mathcal{R}_{12}\bigg (u_k\otimes \:{\begin{array}{|c|}\hline 1\\ \hline 3\\ \hline \end{array}\:}^{\otimes s}\otimes u_1^{\otimes L-s} \bigg )\\
	&=& u_1^{\otimes k}\otimes \: {\begin{array}{|c|}\hline 1\\ \hline 3\\ \hline \end{array}\:}^{\otimes s} \otimes u_1^{\otimes L-k-s} \otimes u_k,
\end{eqnarray*}
and
\begin{eqnarray*}
E_k(p)=-\sum_{i=0}^{L-1}H(u^{(i)}\otimes b_i)=k.
\end{eqnarray*}
Otherwise,
\begin{eqnarray*}
T_k(p)\otimes u_s&=&\mathcal{R}_{L\; L+1}\cdots \mathcal{R}_{23}\mathcal{R}_{12}( u_k\otimes \:{\begin{array}{|c|}\hline 1\\ \hline 3\\ \hline \end{array}\:}^{\otimes s}\otimes u_1^{\otimes L-s})\\
	&=& u_1^{\otimes s}\otimes \:{\begin{array}{|c|}\hline 1\\ \hline 3\\ \hline \end{array}\:}^{\otimes s} \otimes u_1^{\otimes L-2s} \otimes u_k,
\end{eqnarray*}
and
\begin{eqnarray*}
E_s(p)=-\sum_{i=0}^{L-1}H(u^{(i)}\otimes b_i)=s.
\end{eqnarray*}
The result follows because the $\tilde{f}_i,i\neq 0,2$ commute with $T_s$ and preserve $E_s$ (Proposition \ref{commute}).
\end{proof}

We now consider the two-soliton case $$p= \dots [s_1]\dots [s_2]\dots$$ where $s_1 > s_2$.  We can use Proposition \ref{onesol} to associate a two-soliton state $T_r^t(p):=p_t$ at time $t$ with the element $z^{-k_1}b_1 \otimes z^{-k_2}b_2\in \text{Aff}(\widehat{B}^{1,s_1}\times {B}^{1,s_1})\otimes \text{Aff}(\widehat{B}^{1,s_2}\times {B}^{1,s_2}),$ where $k_i:=-\min(r,s_i)t+\gamma_i$, where $\gamma_i$ is the number of `$u_1$'s to the left of $[s_i]$.  If $r>s_2$, then we expect to see the longer soliton catch up with and collide with the shorter one, and, after sufficiently many time steps separate out into another two soliton state.  At the end of this section, we will prove that is the case.  In the following lemmas, we identify elements in  $\text{Aff}(\widehat{B}^{1,s_1}\times {B}^{1,s_1})\otimes \text{Aff}(\widehat{B}^{1,s_2}\times {B}^{1,s_2})$ with their corresponding two-soliton states.

\begin{lemma}[Analogous to Lemma 4.15 in \cite{Yd2}, Lemma 3 in \cite{MOW}]\label{Shift}
Suppose there is a one-soliton state $p=\dots [s]\dots$ corresponding to $z^{-k}((i,s-i),b) \in \text{Aff}(\widehat{B}^{1,s_1}\times {B}^{1,s_1}),0\leq i \leq s.$  Then $T_\natural(p)$ is another one-soliton state, corresponding to $z^{-k}((s,0),b)$, if $i=s$, and $z^{-k-1}((i+1,s-i-1),b)$ otherwise.  We also have $b(p)=\begin{array}{|c|}\hline 1 \\ \hline \end{array}$ if $i=s$ and $b(p)=\begin{array}{|c|}\hline 2 \\ \hline \end{array}$ otherwise. 
\end{lemma}
\begin{proof}
We compute $T_\natural \bigg(\:\begin{array}{|c|}\hline 2\\ \hline b_s\\ \hline\end{array}\:\otimes
\: \begin{array}{|c|}\hline 2\\ \hline b_{s-1}\\ \hline\end{array}\: \otimes
\cdots \otimes
\:\begin{array}{|c|}\hline 2\\ \hline b_{s-j}\\ \hline\end{array}\: \otimes
\:\begin{array}{|c|}\hline 1\\ \hline b_{s-j-1}\\ \hline\end{array}\: \otimes
\:\begin{array}{|c|}\hline 1\\ \hline b_{s-j-2}\\ \hline\end{array}\: \otimes
\cdots
\otimes
\:\begin{array}{|c|}\hline 1\\ \hline b_1\\ \hline\end{array}\: \bigg )$ for the two cases $j=0$ and $j>0$.

\emph{Case 1:} $j=0$.  We compute:
$$\widehat{\mathcal{R}}\bigg (\:\begin{array}{|c|}\hline 1 \\ \hline \end{array}\otimes \begin{array}{|c|}\hline 1 \\ \hline b_k\\ \hline \end{array}\: \bigg )=\begin{array}{|c|}\hline 1 \\ \hline b_k \\ \hline \end{array}\otimes \begin{array}{|c|}\hline 1 \\ \hline \end{array}\:,$$
since $b\neq \{1,\overline{2},\overline{1}\}$.  Therefore, we have $T_\natural(p)=p$ and $b(p)=\begin{array}{|c|}\hline 1 \\ \hline \end{array}\:$.

\emph{Case 2:} $j>0$.  We compute:
\begin{eqnarray*}
\widehat{\mathcal{R}}\bigg (\:\begin{array}{|c|}\hline 1 \\ \hline \end{array}\otimes \begin{array}{|c|}\hline 2 \\ \hline b_s\\ \hline \end{array}\: \bigg )&=&\begin{array}{|c|}\hline 1 \\ \hline 2 \\ \hline \end{array}\otimes \begin{array}{|c|}\hline b_s \\ \hline \end{array}\:,\\
\widehat{\mathcal{R}}\bigg (\:\begin{array}{|c|}\hline b_k \\ \hline \end{array}\otimes \begin{array}{|c|}\hline c \\ \hline b_{k-1}\\ \hline \end{array}\: \bigg )&=&\begin{array}{|c|}\hline c \\ \hline b_k \\ \hline \end{array}\otimes \begin{array}{|c|}\hline b_{k-1}\\ \hline \end{array}\:,
\end{eqnarray*}
where $c=1$ or $2$, and $b_k\leq b_{k-1}.$
\end{proof}

Now consider the two soliton state $z^{-k_1}((s_1,0),b)\otimes z^{-k_2}((i, s_2-i),b')$.  Thanks to Lemma \ref{Shift}, the action of $T_\natural$ has no effect on the first soliton, and therefore we have the following corollary.
\begin{corollary}[Analogous to Lemma 4.15 in \cite{Yd2}, Corollary 1 in \cite{MOW}]\label{2Sol}
Suppose we have a two-soliton state $p=\dots [s_1]\dots \cdots \dots [s_2]\dots$ corresponding to $z^{-k_1}((s_1,0)\times b_1)\otimes z^{-k_2}((i,s_2-i)\times b_2) \in \text{Aff}(\widehat{B}^{1,s_1}\times B^{1,s_1})\otimes \text{Aff}(\widehat{B}^{1,s_2}\times B^{1,s_2}),0\leq i \leq s_2.$  Then $T_\natural(p)$ is another two-soliton state, corresponding to $z^{-k_1}((s_1,0),b_1)\otimes z^{-k_2}((s_2,0),b)$, if $i=s_2$, and $z^{-k_1}((s_1,0),b_1) \otimes z^{-k_2-1}((i+1,s_2-i-1),b_2)$ otherwise.  We also have $b(p)=\begin{array}{|c|}\hline 1\\ \hline \end{array}$ if $i=s_1$ and $b(p)=\begin{array}{|c|}\hline 2\\ \hline \end{array}$ otherwise. 
\end{corollary}
Before we prove the main result on scattering of solitons, we first prove several Lemmas relating $\mathcal{R}$ and $T_r,T_\natural.$
\begin{lemma}[Analogous to Lemma 4.17 in \cite{Yd2}, Lemma 4 in \cite{MOW}]\label{RCommute}
Assume that $s_1>s_2$.  For $p=z^{-k_1}((s_1,0),b_1)\otimes z^{-k_2}((i,s_2-i),b_2)\in \text{Aff}(\widehat{B}^{1,s_1}\times B^{1,s_1} )\otimes \text{Aff}(\widehat{B}^{1,s_2}\times {B}^{1,s_2}),0\leq i \leq s_2$ we have
\begin{enumerate}
\item $T_{\natural}(\widehat{\mathcal{R}}^{\text{Aff}}(p))=\widehat{\mathcal{R}}^{\text{Aff}}(T_{\natural}(p))$, and
\item $b(p)=b(\widehat{\mathcal{R}}^{\text{Aff}}(p)),$
\end{enumerate}
where we use the shifted energy function $\widetilde{H}=2s_2+\widehat{H}$ in $\widehat{\mathcal{R}}^{\text{Aff}}$.
\end{lemma}
\begin{proof}
As in Lemma \ref{Shift} we have the following two cases: $i=s_2$ and $i<s_2$.

\emph{Case 1:} We compute
\begin{eqnarray*}
\lefteqn{T_{\natural}(\widehat{\mathcal{R}}^{\text{Aff}}(z^{-k_1}((s_1,0),b_1)\otimes z^{-k_2}((s_2,0),b_2)))}\qquad\\
&=&T_{\natural}(z^{-k_2+\widetilde{H}}((s_2,0),\tilde{b}_2)\otimes z^{-k_1-\widetilde{H}}((s_1,0),\tilde{b}_1))\\
&=&z^{-k_2+\widetilde{H}}((s_2,0),\tilde{b}_2)\otimes z^{-k_1-\widetilde{H}}((s_1,0),\tilde{b}_1)
\end{eqnarray*}
and,
\begin{eqnarray*}
\lefteqn{\widehat{\mathcal{R}}^{\text{Aff}}(T_{\natural}((z^{-k_1}((s_1,0),b_1)\otimes z^{-k_2}((s_2,0),b_2)))}\qquad\\
&=&\widehat{\mathcal{R}}^{\text{Aff}}(z^{-k_2}((s_2,0),\tilde{b}_2)\otimes z^{-k_1}((s_1,0),\tilde{b}_1))\\
&=&z^{-k_2+\widetilde{H}}((s_2,0),\tilde{b}_2)\otimes z^{-k_1-\widetilde{H}}((s_1,0),\tilde{b}_1)
\end{eqnarray*}
where $\tilde{b}_2\otimes \tilde{b}_1=\mathcal{R}(b_1\otimes b_2).$

\emph{Case 2:}
We compute
\begin{eqnarray*}
\lefteqn{T_{\natural}(\widehat{\mathcal{R}}^{\text{Aff}}(z^{-k_1}((s_1,0),b_1)\otimes z^{-k_2}((i,s_2-i),b_2)))}\qquad\\
&=&T_{\natural}(z^{-k_2+\widetilde{H}}((s_2,0),\tilde{b}_2)\otimes z^{-k_1-\widetilde{H}}((s_1-s_2-i,s_2-i),\tilde{b}_1))\\
&=&z^{-k_2+\widetilde{H}}((s_2,0),\tilde{b}_2)\otimes z^{-k_1-\widetilde{H}-1}((s_1-s_2-i+1,s_2-i-1),\tilde{b}_1)
\end{eqnarray*}
and,
\begin{eqnarray*}
\lefteqn{\widehat{\mathcal{R}}^{\text{Aff}}(T_{\natural}((z^{-k_1}((s_1,0),b_1)\otimes z^{-k_2}((i,s_2-i),b_2)))}\qquad\\
&=&\widehat{\mathcal{R}}^{\text{Aff}}(z^{-k_2}((s_2,0),\tilde{b}_2)\otimes z^{-k_1-1}((i+1,s_2-i-1),\tilde{b}_1))\\
&=&z^{-k_2+\widetilde{H}}((s_2,0),\tilde{b}_2)\otimes z^{-k_1-\widetilde{H}-1}((s_1-s_2-i+1,s_2-i-1),\tilde{b}_1)
\end{eqnarray*}
where $\tilde{b}_2\otimes \tilde{b}_1=\mathcal{R}(b_1\otimes b_2).$
\end{proof}
The proof of the following is the same as that of Lemma 5 in \cite{MOW}.
\begin{lemma}[Analogous to Lemma 4.18 in \cite{Yd2}, Lemma 5 in \cite{MOW}]\label{TCommute}
Let $p\in \mathcal{P}_L,l>0,L\gg 0.$  Then
\begin{enumerate}
\item$T_{\natural}(T_l(p))=T_l(T_\natural(p))$,
\item $b(T_l(p))\otimes u_l=\mathcal{R}(u_l\otimes b(p))$.
\end{enumerate}
\end{lemma}
Now we are ready to prove the main result:
\begin{theorem} \label{th:main}
Let $p=\dots [s_1] \dots [s_2]\dots\in \mathcal{P}_L$ be a two-soliton state with $s_1>s_2$, corresponding to $z^{k_1}b_1\otimes z^{k_2}b_2$ where $k_1,k_2\leq 0$.  Then after sufficiently many applications of $T_r, r>s_2$ the new state is given by 
$$\widehat{\mathcal{R}}^{\text{Aff}}(z^{k_1}b_1\otimes z^{k_2}b_2)=z^{k_2'}\tilde{b}_2\otimes z^{k_1'}\tilde{b}_1$$ 
with phase shift
$$k_2'-k_2=k_1-k_1'=2l_2+\widehat{H}(b_1\otimes b_2).$$
\end{theorem}
\begin{proof} 
By Proposition \ref{commute}, $T_r$ commutes with $\tilde{e}_i, \tilde{f}_i,i\neq 0,2$.  Thus, it is enough to check the scattering rule for the highest weight elements 
$$
z^{k_1}((s_1,0),b)\otimes z^{k_2}((y_1, y_2),b')\in \begin{cases}\text{Aff}((\widehat{B}^{1,s_1})^3)\otimes \text{Aff}((\widehat{B}^{1,s_2})^3), & \text{if }n=4,\\
\text{Aff}(\widehat{B}^{1,s_1}\times \widehat{B}^{2,s_1})\otimes \text{Aff}(\widehat{B}^{1,s_2}\times \widehat{B}^{2,s_2}),&\text{if }n=5,\\
\text{Aff}(\widehat{B}^{1,s_1}\times B^{1,s_1})\otimes \text{Aff}(\widehat{B}^{1,s_2}\times B^{1,s_2})& \text{if }n>5.
\end{cases}.
$$
We will show the statement is true by induction on $y_2$.

Suppose $y_2=0$.  Then $y_1=s_2.$  In this case, the time evolution operator is equivalent to that of the $D_{n-1}^{(1)}$ SCA associated to the crystal $B^{1,s}$ if $n>5$, or the $A_3^{(1)}$-crystal $B^{2,s}$ if $n=4$.
Therefore, by the main results in \cite{HKOTY3} (resp. \cite{Yd1} when $n=4$) the scattering rule is the same as for the $D_{n-1}^{(1)}$ (resp. $A_3^{(1)}$) case, namely:
$$
T_r^{t}(p)=\widehat{\mathcal{R}}^{\text{Aff}}(p),
$$
for sufficiently large $t$.

Assume the statement is true for $<y_2$.  By corollary 1, $T_\natural(p)=z^{k_1}((s_1,0),b)\otimes z^{k_2-1}((y_1+1,y_2-1),b')$, so the inductive assumption holds.  Therefore, $T_r^t(T_\natural(p))$ is a 2-soliton state, for sufficiently large $t$, and 
$$
T_r^t(T_\natural(p))=\widehat{\mathcal{R}}^{\text{Aff}}(T_\natural(p)).
$$
Therefore, by Lemmas \ref{RCommute} and \ref{TCommute}, we have
$$
T_\natural(T_r^t(p))=T_\natural(\widehat{\mathcal{R}}^{\text{Aff}}(p)).
$$
Also, by Lemma \ref{RCommute}, we have $b(\widehat{\mathcal{R}}^{\text{Aff}}(p))=b(p),$ and by Corollary \ref{2Sol} $b(p)=\:\begin{array}{|c|}\hline2\\ \hline \end{array}\:$ since $p$ is a 2-soliton state and $y_2>0$.  By Lemma \ref{TCommute}, $b(T_r(p))\otimes u_r=\mathcal{R}(u_r\otimes b(p))=\mathcal{R}\big (u_r\otimes \:\begin{array}{|c|}\hline2\\ \hline \end{array}\:\big )=\:\begin{array}{|c|}\hline2\\ \hline \end{array}\: \otimes u_r$.  Therefore $b(T_r(p))=\:\begin{array}{|c|}\hline2\\ \hline \end{array}\:$ and, by repeated application of Lemma \ref{RCommute}, we have  $b(T_r^t(p))=\:\begin{array}{|c|}\hline2\\ \hline \end{array}\:$.  Therefore, $b(\widehat{\mathcal{R}}^{\text{Aff}}(p))=b(T_r^t(p))=\:\begin{array}{|c|}\hline2\\ \hline \end{array}\:$, and $T_\natural(T_r^t(p))=T_\natural(\widehat{\mathcal{R}}^{\text{Aff}}(p)).$  Therefore, the crystal maps in \eqref{tn} can be inverted to yield
$$
T_r^t(p)=\widehat{\mathcal{R}}^{\text{Aff}}(p)
$$
is a 2-soliton state.

\end{proof}

\subsection{Examples of scattering of $D_n^{(1)}$-solitons}
Here we give some examples of $D_n^{(1)}$ soliton scattering, for different $n$.

\noindent\emph{Example:} $n=4:$
$$
\begin{array}{cc}
t=0:&
\begin{array}{lllllllllllllllllllllllllll}
 1 & 1 & 1 & 1 & 1 & 2 & 2 & 1 & 1 & 1 & 1 & 1 & 1 & 1 & 1 & 1 & 1 & 1 & 1 & 1 & 1 & 1 & 1 & 1 & 1 & 1 & 1 \\
 \bar{3} & \bar{4} & \bar{4} & 2 & 2 & 4 & 3 & 2 & 2 & 2 & 2 & 2 & 2 & 2 & 2 & 2 & 2 & 2 & 2 & 2 & 2 & 2 & 2 & 2 & 2 & 2 & 2 \\
\end{array}\vspace{1mm} \\
t=1:&
\begin{array}{lllllllllllllllllllllllllll}
 1 & 1 & 1 & 1 & 1 & 1 & 1 & 2 & 2 & 1 & 1 & 1 & 1 & 1 & 1 & 1 & 1 & 1 & 1 & 1 & 1 & 1 & 1 & 1 & 1 & 1 & 1 \\
 2 & 2 & 2 & \bar{3} & \bar{4} & \bar{4} & 2 & 4 & 3 & 2 & 2 & 2 & 2 & 2 & 2 & 2 & 2 & 2 & 2 & 2 & 2 & 2 & 2 & 2 & 2 & 2 & 2 \\
\end{array}\vspace{1mm} \\
t=2:&
\begin{array}{lllllllllllllllllllllllllll}
 1 & 1 & 1 & 1 & 1 & 1 & 1 & 1 & 1 & 2 & 2 & 1 & 1 & 1 & 1 & 1 & 1 & 1 & 1 & 1 & 1 & 1 & 1 & 1 & 1 & 1 & 1 \\
 2 & 2 & 2 & 2 & 2 & 2 & \bar{3} & \bar{4} & \bar{4} & 4 & 3 & 2 & 2 & 2 & 2 & 2 & 2 & 2 & 2 & 2 & 2 & 2 & 2 & 2 & 2 & 2 & 2 \\
\end{array}\vspace{1mm} \\
t=3:&
\begin{array}{lllllllllllllllllllllllllll}
 1 & 1 & 1 & 1 & 1 & 1 & 1 & 1 & 1 & 1 & 1 & \bar{4} & 2 & 1 & 1 & 1 & 1 & 1 & 1 & 1 & 1 & 1 & 1 & 1 & 1 & 1 & 1 \\
 2 & 2 & 2 & 2 & 2 & 2 & 2 & 2 & 2 & \bar{3} & \bar{4} & 4 & 3 & 2 & 2 & 2 & 2 & 2 & 2 & 2 & 2 & 2 & 2 & 2 & 2 & 2 & 2 \\
\end{array}\vspace{1mm} \\
t=4:&
\begin{array}{lllllllllllllllllllllllllll}
 1 & 1 & 1 & 1 & 1 & 1 & 1 & 1 & 1 & 1 & 1 & 1 & 1 & \bar{4} & 2 & 1 & 1 & 1 & 1 & 1 & 1 & 1 & 1 & 1 & 1 & 1 & 1 \\
 2 & 2 & 2 & 2 & 2 & 2 & 2 & 2 & 2 & 2 & 2 & 2 & \bar{3} & \bar{3} & 3 & 3 & 2 & 2 & 2 & 2 & 2 & 2 & 2 & 2 & 2 & 2 & 2 \\
\end{array}\vspace{1mm} \\
t=5:&
\begin{array}{lllllllllllllllllllllllllll}
 1 & 1 & 1 & 1 & 1 & 1 & 1 & 1 & 1 & 1 & 1 & 1 & 1 & 1 & 1 & 1 & 2 & 2 & 1 & 1 & 1 & 1 & 1 & 1 & 1 & 1 & 1 \\
 2 & 2 & 2 & 2 & 2 & 2 & 2 & 2 & 2 & 2 & 2 & 2 & 2 & 2 & \bar{3} & \bar{4} & \bar{3} & 3 & 3 & 2 & 2 & 2 & 2 & 2 & 2 & 2 & 2 \\
\end{array}\vspace{1mm} \\
t=6:&
\begin{array}{lllllllllllllllllllllllllll}
 1 & 1 & 1 & 1 & 1 & 1 & 1 & 1 & 1 & 1 & 1 & 1 & 1 & 1 & 1 & 1 & 1 & 1 & 1 & 2 & 2 & 1 & 1 & 1 & 1 & 1 & 1 \\
 2 & 2 & 2 & 2 & 2 & 2 & 2 & 2 & 2 & 2 & 2 & 2 & 2 & 2 & 2 & 2 & \bar{3} & \bar{4} & 2 & \bar{3} & 3 & 3 & 2 & 2 & 2 & 2 & 2 \\
\end{array}\vspace{1mm} \\
t=7:&
\begin{array}{lllllllllllllllllllllllllll}
 1 & 1 & 1 & 1 & 1 & 1 & 1 & 1 & 1 & 1 & 1 & 1 & 1 & 1 & 1 & 1 & 1 & 1 & 1 & 1 & 1 & 1 & 2 & 2 & 1 & 1 & 1 \\
 2 & 2 & 2 & 2 & 2 & 2 & 2 & 2 & 2 & 2 & 2 & 2 & 2 & 2 & 2 & 2 & 2 & 2 & \bar{3} & \bar{4} & 2 & 2 & \bar{3} & 3 & 3 & 2 & 2 \\
\end{array}
\end{array}
$$
The initial state corresponds to $z^0((3,0),(0,3),(2,1))\otimes z^{-5}(0,2),(2,0),(1,1)).$  We compute the $R$-matrix of $(A_1^{(1)})^{\oplus 3}$-crystals as follows: $\mathcal{R}^{\text{aff}}(z^0((3,0),(0,3),(2,1))\otimes z^{-5}(0,2),(2,0),(1,1)))=z^{-4}((2,0),(1,1),(0,2))\otimes z^{-1}((1,2),(2,1),(2,1))$. According to Theorem \ref{th:main} we expect this to correspond to the final state, which is the case.

\noindent\emph{Example:} $n=5:$
$$
\begin{array}{cc}
t=0:&\begin{array}{lllllllllllllllllllllllllll}
 2 & 2 & 1 & 1 & 1 & 1 & 2 & 1 & 1 & 1 & 1 & 1 & 1 & 1 & 1 & 1 & 1 & 1 & 1 & 1 & 1 & 1 & 1 & 1 & 1 & 1 & 1 \\
 \bar{3} & 5 & 4 & 3 & 2 & 2 & \bar{4} & \bar{5} & 2 & 2 & 2 & 2 & 2 & 2 & 2 & 2 & 2 & 2 & 2 & 2 & 2 & 2 & 2 & 2 & 2 & 2 & 2 \\
\end{array}\vspace{1mm} \\
t=1:&\begin{array}{lllllllllllllllllllllllllll}
 1 & 1 & 1 & 1 & 2 & 2 & 1 & 1 & 2 & 1 & 1 & 1 & 1 & 1 & 1 & 1 & 1 & 1 & 1 & 1 & 1 & 1 & 1 & 1 & 1 & 1 & 1 \\
 2 & 2 & 2 & 2 & \bar{3} & 5 & 4 & 3 & \bar{4} & \bar{5} & 2 & 2 & 2 & 2 & 2 & 2 & 2 & 2 & 2 & 2 & 2 & 2 & 2 & 2 & 2 & 2 & 2 \\
\end{array}\vspace{1mm} \\
t=2:&\begin{array}{lllllllllllllllllllllllllll}
 1 & 1 & 1 & 1 & 1 & 1 & 1 & 1 & 2 & 2 & 4 & 1 & 1 & 1 & 1 & 1 & 1 & 1 & 1 & 1 & 1 & 1 & 1 & 1 & 1 & 1 & 1 \\
 2 & 2 & 2 & 2 & 2 & 2 & 2 & 2 & \bar{3} & 5 & \bar{4} & \bar{5} & 3 & 2 & 2 & 2 & 2 & 2 & 2 & 2 & 2 & 2 & 2 & 2 & 2 & 2 & 2 \\
\end{array}\vspace{1mm} \\
t=3:&\begin{array}{lllllllllllllllllllllllllll}
 1 & 1 & 1 & 1 & 1 & 1 & 1 & 1 & 1 & 1 & 1 & 2 & 1 & 2 & 2 & 1 & 1 & 1 & 1 & 1 & 1 & 1 & 1 & 1 & 1 & 1 & 1 \\
 2 & 2 & 2 & 2 & 2 & 2 & 2 & 2 & 2 & 2 & 2 & 5 & 4 & \bar{3} & \bar{4} & \bar{5} & 3 & 2 & 2 & 2 & 2 & 2 & 2 & 2 & 2 & 2 & 2 \\
\end{array}\vspace{1mm} \\
t=4:&\begin{array}{lllllllllllllllllllllllllll}
 1 & 1 & 1 & 1 & 1 & 1 & 1 & 1 & 1 & 1 & 1 & 1 & 1 & 2 & 1 & 1 & 1 & 2 & 2 & 1 & 1 & 1 & 1 & 1 & 1 & 1 & 1 \\
 2 & 2 & 2 & 2 & 2 & 2 & 2 & 2 & 2 & 2 & 2 & 2 & 2 & 5 & 4 & 2 & 2 & \bar{3} & \bar{4} & \bar{5} & 3 & 2 & 2 & 2 & 2 & 2 & 2 \\
\end{array}
\end{array}
$$
In this case, the initial state corresponds to $z^0\bigg ((2,2),\begin{pmatrix} 2&1&1&0\\ 0&1&2&1 \end{pmatrix}\bigg )\otimes z^{-6}\bigg ((1,1),\\\begin{pmatrix} 1&1&0&0 \\0&0&0&2 \end{pmatrix} \bigg ).$  We compute the $R$-matrix for $A_1^{(1)}\oplus A_3^{(1)}$-crystals as follows: $\mathcal{R}^{\text{Aff}}\big (z^0\bigg ((2,2),\begin{pmatrix} 2&1&1&0\\ 0&1&2&1 \end{pmatrix}\bigg )\otimes z^{-6}\bigg ((1,1),\begin{pmatrix} 1&1&0&0 \\0&0&0&2 \end{pmatrix}\bigg ) \big )=z^{-5}\bigg ((1,1),\begin{pmatrix} 1&1&0&0&\\0&0&2&0 \end{pmatrix} \bigg )\otimes z^{-1}\bigg ((2,2),\begin{pmatrix} 2&1&1&0\\0&1&0&3 \end{pmatrix} \bigg )$, which again confirms Theorem \ref{th:main}.

\noindent \emph{Example:} $n=6:$
$$
\begin{array}{cc}
t=0:&\begin{array}{lllllllllllllllllllllllllll}
 2 & 2 & 1 & 1 & 1 & 1 & 1 & 2 & 2 & 1 & 1 & 1 & 1 & 1 & 1 & 1 & 1 & 1 & 1 & 1 & 1 & 1 & 1 & 1 & 1 & 1 & 1 \\
 \bar{3} & \bar{5} & 6 & 5 & 4 & 2 & 2 & \bar{5} & \bar{5} & 2 & 2 & 2 & 2 & 2 & 2 & 2 & 2 & 2 & 2 & 2 & 2 & 2 & 2 & 2 & 2 & 2 & 2 \\
\end{array}\vspace{1mm} \\
t=1:&\begin{array}{lllllllllllllllllllllllllll}
 1 & 1 & 1 & 1 & 1 & 2 & 2 & 1 & 1 & 4 & 2 & 1 & 1 & 1 & 1 & 1 & 1 & 1 & 1 & 1 & 1 & 1 & 1 & 1 & 1 & 1 & 1 \\
 2 & 2 & 2 & 2 & 2 & \bar{3} & \bar{5} & 6 & 5 & \bar{5} & \bar{5} & 2 & 2 & 2 & 2 & 2 & 2 & 2 & 2 & 2 & 2 & 2 & 2 & 2 & 2 & 2 & 2 \\
\end{array}\vspace{1mm} \\
t=2:&\begin{array}{lllllllllllllllllllllllllll}
 1 & 1 & 1 & 1 & 1 & 1 & 1 & 1 & 1 & 1 & 2 & \bar{5} & 6 & 2 & 1 & 1 & 1 & 1 & 1 & 1 & 1 & 1 & 1 & 1 & 1 & 1 & 1 \\
 2 & 2 & 2 & 2 & 2 & 2 & 2 & 2 & 2 & 2 & \bar{3} & \bar{4} & \bar{5} & 4 & 4 & 2 & 2 & 2 & 2 & 2 & 2 & 2 & 2 & 2 & 2 & 2 & 2 \\
\end{array}\vspace{1mm} \\
t=3:&
\begin{array}{lllllllllllllllllllllllllll}
 1 & 1 & 1 & 1 & 1 & 1 & 1 & 1 & 1 & 1 & 1 & 1 & 1 & 1 & 1 & 2 & 2 & 2 & 2 & 1 & 1 & 1 & 1 & 1 & 1 & 1 & 1 \\
 2 & 2 & 2 & 2 & 2 & 2 & 2 & 2 & 2 & 2 & 2 & 2 & 2 & \bar{5} & 6 & \bar{3} & \bar{4} & \bar{5} & 4 & 4 & 2 & 2 & 2 & 2 & 2 & 2 & 2 \\
\end{array}\vspace{1mm} \\
t=4:&\begin{array}{lllllllllllllllllllllllllll}
 1 & 1 & 1 & 1 & 1 & 1 & 1 & 1 & 1 & 1 & 1 & 1 & 1 & 1 & 1 & 1 & 1 & 1 & 1 & 1 & 2 & 2 & 2 & 2 & 1 & 1 & 1 \\
 2 & 2 & 2 & 2 & 2 & 2 & 2 & 2 & 2 & 2 & 2 & 2 & 2 & 2 & 2 & \bar{5} & 6 & 2 & 2 & 2 & \bar{3} & \bar{4} & \bar{5} & 4 & 4 & 2 & 2 \\
\end{array}
\end{array}
$$
In this example, the initial state corresponds to $z^0((3,2),(0,1,1,1,0,1,0,1))\otimes z^{-7}((0,2),(0,0,0,0,0,2,0,0)).$  We compute the combinatorial $R$-matrix for $A_1^{(1)}\oplus D_4^{(1)}$-crystals as follows: $\mathcal{R}^{\text{Aff}}(z^0((3,2),(0,1,1,1,0,1,0,1))\otimes z^{-7}((0,2),(0,0,0,0,0,\\2,0,0)))=z^{-7}((2,0),(0,0,0,1,0,1,0,0))\otimes z^0((1,4),(0,2,0,0,0,1,1,1))$.  This is in agreement with Theorem \ref{th:main}.
\section{Acknowledgments}
KCM was partially supported by the NSA grant H98230-12-1-0248.
This work was done in part during the visit of EAW to the University of S\~{a}o Paulo as a postdoctoral fellow.  This author is grateful to the University of S\~{a}o Paulo for their hospitality and to FAPESP for financial support (2011/12079-5).

\end{document}